\documentclass[12pt,a4paper]{amsart}
\usepackage{amssymb,eucal}
\usepackage[all,cmtip]{xy}

\calclayout

\newcommand{\+}{\nobreakdash-}
\renewcommand{\:}{\colon}

\newcommand{\rarrow}{\longrightarrow}

\newcommand{\ot}{\otimes}

\DeclareMathOperator{\Hom}{Hom}
\DeclareMathOperator{\Ext}{Ext}
\DeclareMathOperator{\Cohom}{Cohom}
\DeclareMathOperator{\im}{im}
\DeclareMathOperator{\coker}{coker}

\newcommand{\Modl}{{\operatorname{\mathsf{--Mod}}}}
\newcommand{\Modr}{{\operatorname{\mathsf{Mod--}}}}
\newcommand{\Modrlf}{{\operatorname{\mathsf{Mod_\lf--}}}}
\newcommand{\Contra}{{\operatorname{\mathsf{--Contra}}}}
\newcommand{\Contrar}{{\operatorname{\mathsf{Contra--}}}}
\newcommand{\Comodl}{{\operatorname{\mathsf{--Comod}}}}
\newcommand{\Comodr}{{\operatorname{\mathsf{Comod--}}}}
\newcommand{\Comodrlf}{{\operatorname{\mathsf{Comod_\lf--}}}}
\newcommand{\comodl}{{\operatorname{\mathsf{--comod}}}}
\newcommand{\comodr}{{\operatorname{\mathsf{comod--}}}}
\newcommand{\contra}{{\operatorname{\mathsf{--contra}}}}
\newcommand{\Vect}{{\operatorname{\mathsf{--Vect}}}}
\newcommand{\vect}{{\operatorname{\mathsf{--vect}}}}

\newcommand{\Ab}{\mathsf{Ab}}
\newcommand{\Br}{\mathsf{Br}}
\newcommand{\TL}{\mathsf{TL}}

\newcommand{\lf}{{\mathsf{lf}}}

\newcommand{\sop}{{\mathsf{op}}}
\newcommand{\id}{{\mathrm{id}}}

\newcommand{\lrarrow}{\mskip.5\thinmuskip\relbar\joinrel\relbar
   \joinrel\rightarrow\mskip.5\thinmuskip\relax}

\newcommand{\bigtimes}{\mathop{\text{\huge $\times$}}}

\newcommand{\oc}{\mathbin{\text{\smaller$\square$}}}

\newcommand{\fK}{\mathfrak K}
\newcommand{\fP}{\mathfrak P}
\newcommand{\fQ}{\mathfrak Q}
\newcommand{\fR}{\mathfrak R}
\newcommand{\fS}{\mathfrak S}
\newcommand{\fT}{\mathfrak T}
\newcommand{\fW}{\mathfrak W}

\newcommand{\LL}{\mathcal L}
\newcommand{\M}{\mathcal M}
\newcommand{\N}{\mathcal N}
\newcommand{\C}{\mathcal C}
\newcommand{\D}{\mathcal D}
\newcommand{\B}{\mathcal B}

\newcommand{\sE}{\mathsf E}

\newcommand{\boZ}{\mathbb Z}

\newcommand{\Section}[1]{\bigskip\section{#1}\medskip}
\theoremstyle{plain}
\newtheorem{thm}{Theorem}[section]
\newtheorem{prop}[thm]{Proposition}
\newtheorem{lem}[thm]{Lemma}
\newtheorem{cor}[thm]{Corollary}
\theoremstyle{definition}
\newtheorem{rem}[thm]{Remark}
\newtheorem{ex}[thm]{Example}
\newtheorem{exs}[thm]{Examples}
\newtheorem{defn}[thm]{Definition}
\newtheorem{constr}[thm]{Construction}
\newtheorem{qst}[thm]{Question}

\begin{document}

\title{Comodules and contramodules over coalgebras \\
associated with locally finite categories}

\author{Leonid Positselski}

\address{Institute of Mathematics, Czech Academy of Sciences \\
\v Zitn\'a~25, 115~67 Prague~1 \\ Czech Republic} 

\email{positselski@math.cas.cz}

\begin{abstract}
 We explain how to attach a coalgebra $\C$ over a field~$k$ to a small
$k$\+linear category $\sE$ satisfying suitable finiteness conditions.
 In this context, we study full-and-faithfulness of the contramodule
forgetful functor, and describe explicitly the categories of locally
finite left $\C$\+comodules and left $\C$\+contramodules as certain
full subcategories of the category of left $\sE$\+modules.
\end{abstract}

\maketitle

\tableofcontents

\section*{Introduction}
\medskip

 Let $\sE$ be a small preadditive category.
 Then \emph{left\/ $\sE$\+modules} are defined as covariant additive
functors $\sE\rarrow\Ab$, while \emph{right\/ $\sE$\+modules} are
the contravariant additive functors $\sE^\sop\rarrow\Ab$ (where $\Ab$
is the category of abelian groups).
 The category of left $\sE$\+modules $\sE\Modl$ is a Grothendieck
abelian category with a set of small projective generators indexed by
the objects of the category~$\sE$.
 However, when $\sE$ has infinitely many nonzero objects, the category
$\sE\Modl$ does \emph{not} admit a single small projective generator.
 So the category $\sE\Modl$ is similar to, but different from
the usual categories of modules over (associative, unital) rings.
 One speaks of $\sE\Modl$ as the category of modules over the ``ring
with several objects'' $\sE$, or as the category of modules over
a ``locally unital'' (but nonunital) ring $R_\sE=\bigoplus_{x,y\in\sE}
\Hom_\sE(x,y)$ associated with~$\sE$.

 In the language of topological algebra, one can always describe
the abelian category $\sE\Modl$ as the category of discrete modules,
or if one prefers, the category of contramodules over the topological
ring $\fR_\sE$ associated with~$\sE$.
 There are two versions of $\fR_\sE$, the left and the right one.
 Let us make out choice and put $\fR_\sE=
\prod_{y\in\sE}\bigoplus_{x\in\sE}\Hom_\sE(x,y)$.
 Then $\fR_\sE$ is the ring of infinite, row-finite matrices whose
entries are morphisms in~$\sE$.
 We endow $\fR_\sE$ with the obvious topology in which the rows
$\bigoplus_{x\in\sE}\Hom_\sE(x,y)$ are discrete abelian groups, while
the whole ring $\fR_\sE$ carries the product topology of the discrete
topologies of the rows.

 Then one can easily see that the category of discrete right
$\fR_\sE$\+modules is naturally equivalent to the category of
right $\sE$\+modules $\Modr\sE$.
 On the other hand, applying the result of~\cite[Corollary~6.3 and
Theorem~7.1]{PS1} to the projective generator $\bigoplus_{x\in\sE}
\Hom_{\sE}(x,{-})$ of the category of left $\sE$\+modules $\sE\Modl$
shows that $\sE\Modl$ is naturally equivalent to the category of
left $\fR_\sE$\+contramodules.

 In this paper, we take a different approach and assign (discrete)
\emph{coalgebras} over the field~$k$ to some $k$\+linear categories.
 For any coalgebra $\C$ over~$k$, the dual vector space $\C^*$ has
a natural structure of topological algebra over~$k$; but most
topological algebras do not arise in this way.
 Furthermore, the dual vector space to an infinite-dimensional algebra
usually does not have a coalgebra structure.
 So one needs to impose some restrictions on a $k$\+linear category
$\sE$ in order to produce a coalgebra from it.
 We say that $\sE$ is \emph{locally finite} if two conditions are
satisfied: the $k$\+vector space $\Hom_\sE(x,y)$ is finite-dimensional
for all objects $x$, $y\in\sE$, and the set of all objects~$z$ for
which there exist finite chains of composable nonzero morphisms
$x\rarrow\dotsb\rarrow z$ and $z\rarrow\dotsb\rarrow y$ is
finite for all $x$, $y\in\sE$.
 For a locally finite $k$\+linear category $\sE$, we construct
a coassociative, counital coalgebra~$\C_\sE$.
 Let us warn the reader that the dual topological algebra $\C_\sE^*$
is still \emph{not} isomorphic to the topological algebra $\fR_\sE$ in
most cases.

 Given a $k$\+linear category $\sE$, additive functors $\sE\rarrow\Ab$
are the same things as $k$\+linear functors $\sE\rarrow k\Vect$ (where
$k\Vect$ is the category of $k$\+vector spaces).
 So one can interpret left $\sE$\+modules as covariant $k$\+linear
functors $\sE\rarrow k\Vect$.
 We explain that to every left $\C_\sE$\+comodule, as well as to every
left $\C_\sE$\+contramodule, one can naturally assign a left
$\sE$\+module.
 So there are forgetful functors $\C_\sE\Comodl\rarrow\sE\Modl$ and
$\C_\sE\Contra\rarrow\sE\Modl$.
 The former functor is always fully faithful; so it is perhaps better
to call it the comodule \emph{inclusion} functor.
 Using the technique originally developed for conilpotent coalgebras in
the paper~\cite[Theorem~2.1]{Psm}, we spell out a sufficient condition
on the category $\sE$ (which we call \emph{lower strict local
finiteness}) implying that the contramodule forgetful functor is fully
faithful as well.
 This is the main result of a longish
Section~\ref{contramodule-full-and-faithfulness-secn};
see Theorem~\ref{main-full-and-faithfulness-theorem}.

 A left $\sE$\+module $N$ is said to be \emph{locally finite} if
the corresponding $k$\+linear functor takes values in the category
of finite-dimensional vector spaces $k\vect$, that is
$N\:\sE\rarrow k\vect$.
 A left $\C_\sE$\+comodule or a left $\C_\sE$\+contramodule is called
\emph{locally finite} if the corresponding left $\sE$\+module is
locally finite.
 The main results of this paper
(Theorems~\ref{locally-finite-comodules-described}
and~\ref{locally-finite-contramodules-described}) describe the full
subcategories of locally finite left $\C_\sE$\+comodules and locally
finite left $\C_\sE$\+contramodules inside the ambient abelian
category of left $\sE$\+modules.
 Specifically, a locally finite left $\sE$\+module $M$ arises from
a left $\C_\sE$\+comodule if and only if, for every object $x\in\sE$,
the set of all objects $y\in\sE$ with nonzero action map
$\Hom_\sE(x,y)\ot_k M(x)\rarrow M(y)$ is finite.
 A locally finite left $\sE$\+module $P$ arises from a left
$\C_\sE$\+contramodule if and only if, for every object $y\in\sE$,
the set of all objects $x\in\sE$ with nonzero action map
$\Hom_\sE(x,y)\ot_k P(x)\rarrow P(y)$ is finite.
 The comodule case is rather easy; the proof of the contramodule
version is much more involved and depends on the assumption of
lower strict local finiteness of the $k$\+linear category~$\sE$.

 It is clear that, for any $k$\+linear category $\sE$, the vector
space dualization functor $N\longmapsto P=N^*$, \ $P(x)=N(x)^*=
\Hom_k(N(x),k)$ for all $x\in\sE$, provides an anti-equivalence between
the abelian categories of locally finite right and left $\sE$\+modules.
 It follows from the main results of this paper that the vector
space dualization functor $\N\longmapsto\fP=\N^*=\Hom_k(\N,k)$
also provides an anti-equivalence between the abelian categories of
locally finite right $\C_\sE$\+comodules and locally finite
left $\C_\sE$\+contramodules for a lower strictly locally finite
$k$\+linear category~$\sE$.
 This observation is made in Corollary~\ref{anti-equivalence-corollary}.

 Given a coalgebra $\C$ over a field~$k$, is the essential image of
the comodule inclusion functor $\C\Comodl\rarrow\C^*\Modl$ closed
under extensions?
 More generally, given a dense subring $R\subset\C^*$, is
the essential image of the inclusion functor $\C\Comodl\rarrow
R\Modl$ closed under extensions?
 There is a vast body of literature on this question, including
such papers as~\cite{Rad,Shu,Lin,CNO,Cua,TT,Iov}.
 In particular, for a conilpotent coalgebra $\C$ (called ``pointed
irreducible'' in the traditional terminology of~\cite{Swe,Mon}),
the full subcategory $\C\Comodl$ is closed under extensions
in $\C^*\Modl$ if and only if the full subcategory $\Comodr\C$
is closed under extensions in $\Modr\C^*$, and if and only if
the coalgebra $\C$ is finitely cogenerated~\cite[Corollary~2.4 and
Section~2.5]{Rad}, \cite[Theorem~4.6]{Shu}, \cite[Corollary~21]{Lin},
\cite[Theorem~2.8]{CNO}, \cite[Proposition~3.13 and
Corollary~3.14]{Cua}, \cite[Lemma~1.2 and Theorem~4.8]{Iov}.
 The same property characterizes the conilpotent coalgebras $\C$ for
which the contramodule forgetful functor $\C\Contra\rarrow
\C^*\Modl$ is fully faithful~\cite[Theorem~2.1]{Psm},
\cite[Example~7.2]{Phff}.

 In Section~\ref{extension-closure-secn} of this paper, we show that
the full subcategory $\Comodr\C_\sE$ is closed under extensions in
$\Modr\sE$ for any lower strictly locally finite $k$\+linear
category~$\sE$.
 This observation is not really new, but it complements the main results
of this paper nicely; so we include it for the benefit of the reader.
 Under the same assumptions, the full subcategory of \emph{locally
finite} left $\C_\sE$\+contramodules is closed under extensions in
the category of (locally finite or arbitrary) left $\sE$\+modules.
 Based on the discussion of the conilpotent case in~\cite{Phff}
(see~\cite[Theorems~6.2 and~7.1  for $n=2$]{Phff}) we expect that
the lower strict local finiteness assumption is \emph{not} sufficient
to make the whole category of (locally infinite) left
$\C_\sE$\+contramodules closed under extensions in $\sE\Modl$.

 In the final Section~\ref{upper-lower-secn}, we discuss
\emph{upper finite} $k$\+linear categories $\sE$, for which
the abelian category of left $\C_\sE$\+comodules is always equivalent
to the abelian category of left $\sE$\+modules.
 Dual-analogously, for a \emph{lower finite} $k$\+linear category
$\sE$, the abelian category of left $\C_\sE$\+contramodules is
equivalent to the abelian category of left $\sE$\+modules.
 For right comodules and right contramodules, the roles of
the ``upper'' and ``lower'' are reversed.
 These special cases are much simpler than those of arbitrary
locally finite $k$\+linear categories or lower strictly
locally finite $k$\+linear categories $\sE$ considered in
the main results of the paper.

 One of the aims of the present paper is to serve as a background
reference source for the paper~\cite{Psa}.
 In the present paper we study coalgebras associated with $k$\+linear
categories, and particularly comodules and contramodules over such
coalgebras.
 In the paper~\cite{Psa}, \emph{semialgebras} arising from $k$\+linear
categories (and more generally from nonunital $k$\+algebras) are
constructed, and \emph{semimodules} and \emph{semicontramodules} over
such semialgebras are considered.
 In addition, \cite[Section~6]{Psa} offers a detailed discussion of
some examples which motivated both the present paper and~\cite{Psa}.

\subsection*{Acknowledgement}
 I~am grateful to Catharina Stroppel for a kind invitation to Bonn
and a stimulating discussion which prompted me to think about comodules
and contramodules over $k$\+linear categories.
 I~also wish to thank Michal Hrbek and Jan \v St\!'ov\'\i\v cek for
very helpful conversations.
 The author is supported by the GA\v CR project 23-05148S and
the Czech Academy of Sciences (RVO~67985840).

\Section{Preliminaries on Coalgebras, Comodules, and Contramodules}
\label{preliminaries-secn}

 All \emph{coalgebras} in this paper are coassociative and, unless
otherwise mentioned, counital.
 All coalgebra homomorphisms are presumed to agree with the counits,
and all co/contramodules are co/contraunital by default.
 The books~\cite{Swe,Mon} are classical reference sources on coalgebras.
 An introductory discussion can be found in the present author's
survey~\cite[Section~3]{Pksurv}.

 A (coassociative, counital) \emph{coalgebra} $\C$ over a field $k$ is
a $k$\+vector space endowed with $k$\+linear maps of
\emph{comultiplication} $\mu\:\C\rarrow\C\ot_k\C$ and \emph{counit}
$\epsilon\:\C\rarrow k$ satisfying the usual coassociativity and
counitality axioms.
 The latter are expressed by the diagrams
$$
 \C\rarrow\C\ot_k\C\rightrightarrows\C\ot_k\C\ot_k\C
$$
and
$$
 \C\rarrow\C\ot_k\C\rightrightarrows\C.
$$
 Here, on the former diagram, the two compositions of
the comultiplication map $\mu\:\C\rarrow\C\ot_k\C$ with the two maps
$\C\ot_k\C\rightrightarrows\C\ot_k\C\ot_k\C$ induced by
the comultiplication map must be equal to each other.
 On the latter diagram, both the compositions of the comultiplication
map with the two maps $\C\ot_k\C\rightrightarrows\C$ induced by
the counit map $\epsilon\:\C\rarrow k$ must be equal to the identity
map~$\id_\C$.

 A \emph{left comodule} $\M$ over $\C$ is a $k$\+vector space endowed
with a $k$\+linear map of \emph{left\/ $\C$\+coaction} $\nu\:\M\rarrow
\C\ot_k\M$ satisfying the coassociativity and counitality axioms.
 The latter are expressed by the diagrams
$$
 \M\rarrow\C\ot_k\M\rightrightarrows\C\ot_k\C\ot_k\M
$$
and
$$
 \M\rarrow\C\ot_k\M\rarrow\M.
$$
 Here, on the former diagram, the two compositions of the coaction
map $\nu\:\M\rarrow\C\ot_k\M$ with the two maps
$\M\ot_k\C\rightrightarrows\C\ot_k\C\ot_k\M$ induced by
the coaction and comultiplication maps~$\nu$ and~$\mu$ must be equal to
each other.
 On the latter diagram, the composition of the coaction map with
the map $\C\ot_k\M\rarrow\M$ induced by the counit map must be equal to
the identity map~$\id_\M$.
 A \emph{right comodule} $\N$ over $\C$ is a $k$\+vector space endowed
with a \emph{right coaction} map $\nu\:\N\rarrow\N\ot_k\C$ satisfying
the similar coassociativity and counitality axioms.

 The notion of a contramodule over a coalgebra goes back to Eilenberg
and Moore~\cite[Section~III.5]{EM}.
 For more recent references, see~\cite[Sections~0.2.4 or~3.1.1]{Psemi},
\cite[Sections~1.1\+-1.6]{Prev}, or~\cite[Section~8]{Pksurv}.

 A \emph{left contramodule} $\fP$ over $\C$ is a $k$\+vector space
endowed with a $k$\+linear map of \emph{left\/ $\C$\+contraaction}
$\pi\:\Hom_k(\C,\fP)\rarrow\fP$ satisfying
the \emph{contraassociativity} and \emph{contraunitality} axioms.
 The latter are expressed by the diagrams
\begin{equation} \label{contraassociativity}
 \Hom_k(\C,\Hom_k(\C,\fP))\simeq\Hom_k(\C\ot_k\C,\>\fP)
 \,\rightrightarrows\,\Hom_k(\C,\fP)\rarrow\fP
\end{equation}
and
\begin{equation} \label{contraunitality}
 \fP\rarrow\Hom_k(\C,\fP)\rarrow\fP.
\end{equation}

 Here, on the contraassociativity diagram~\eqref{contraassociativity},
the map $\Hom_k(\mu,\fP)\:\Hom_k(\C\ot_k\C,\>\fP)\rarrow\Hom_k(\C,\fP)$
is induced by the comultiplication map $\mu\:\C\rarrow\C\ot_k\C$, while
the map $\Hom_k(\C,\pi)\:\Hom_k(\C,\Hom_k(\C,\fP))\rarrow\Hom_k(\C,\fP)$
is induced by the contraaction map $\pi\:\Hom_k(\C,\fP)\rarrow\fP$.
 The identification $\Hom_k(\C,\Hom_k(\C,\fP))\simeq
\Hom_k(\C\ot_k\C,\>\fP)$ is obtained as a particular case of
the natural isomorphism $\Hom_k(V,\Hom_k(U,W))\simeq
\Hom_k(U\ot_kV,\>W)$, which holds for all $k$\+vector spaces $U$, $V$,
and~$W$.
 The two compositions of the resulting two parallel maps with
the contraaction map~$\pi$ must be equal to each other.

 On the contraunitality diagram~\eqref{contraunitality}, the composition
of the map $\Hom_k(\epsilon,\fP)\:\fP\rarrow\Hom_k(\C,\fP)$ induced by
the counit map $\epsilon\:\C\rarrow k$ with the contraaction map
$\pi\:\Hom_k(\C,\fP)\rarrow\fP$ must be equal to the identity
map~$\id_\fP$.
 The definition of a \emph{right contramodule} $\fQ$ over $\C$ is
similar, with the only difference that the identification
$\Hom_k(\C,\Hom_k(\C,\fQ))\simeq\Hom_k(\C\ot_k\C,\>\fQ)$ arising
as a particular case of the natural isomorphism $\Hom_k(V,\Hom_k(U,W))
\simeq\Hom_k(V\ot_kU,\>W)$ is used.

 For any right $\C$\+comodule $\N$ and any $k$\+vector space $V$,
the vector space $\fP=\Hom_k(\N,V)$ has a natural structure of
left $\C$\+contramodule.
 The contraaction map
\begin{multline*}
 \pi_\fP=\Hom_k(\nu_\N,V)\: \\
 \Hom_k(\C,\Hom_k(\N,V))\simeq\Hom_k(\N\ot_k\C,\>V)
 \lrarrow\Hom_k(\N,V)
\end{multline*}
is obtained by applying the contravariant functor $\Hom_k({-},V)$ to
the coaction map $\nu_\N\:\N\rarrow\N\ot_k\C$.

 The categories $\C\Comodl$ and $\Comodr\C$ of left and right
$\C$\+comodules are Grothendieck abelian categories; so they have
enough injective objects.
 A left $\C$\+comodule is injective if and only if it is a direct
summand of a \emph{cofree} left $\C$\+comodule, i.~e., a left
$\C$\+comodule of the form $\C\ot_k V$, where $V\in k\Vect$.
 For any left $\C$\+comodule $\LL$, the vector space of
$\C$\+comodule morphisms $\LL\rarrow\C\ot_kV$ is naturally isomorphic
to the vector space of $k$\+linear maps $\LL\rarrow V$.
{\hbadness=1250\par}

 The categories $\C\Contra$ and $\Contrar\C$ of left and right
$\C$\+contramodules are locally presentable abelian categories with
enough projective objects.
 A left $\C$\+contramodule is projective if and only if it is
a direct summand of a \emph{free} left $\C$\+contramodule, i.~e.,
a left $\C$\+contramodule of the form $\Hom_k(\C,V)$, where
$V\in k\Vect$.
 For any left $\C$\+contramodule $\fQ$, the vector space of
$\C$\+contramodule morphisms $\Hom_k(\C,V)\rarrow\fQ$ is naturally
isomorphic to the vector space of $k$\+linear maps $V\rarrow\fQ$.

 The forgetful functor $\C\Comodl\rarrow k\Vect$ is exact and
faithful.
 It preserves infinite coproducts, but \emph{not} infinite products
(except when $\C$ is finite-dimensional).
 The forgetful functor $\C\Contra\rarrow k\Vect$ is exact and
faithful as well.
 It preserves infinite products, but \emph{not} infinite coproducts.

 Occassionally we will consider coalgebras without counit, defined
in the obvious way (drop the counit map~$\epsilon$ and the counitality
axiom from the definitions above).
 The definitions of comodules and contramodules over a noncounital
coalgebra are obtained by dropping the co/contraunitality axiom in
the definitions of co/contramodules given above.

 A noncounital coalgebra $\D$ over a field~$k$ is said to be
\emph{conilpotent} if for every element $d\in D$ there exists
an integer $n\ge1$ such that the iterated comultiplication
map $\mu^{(n)}\:\D\rarrow\D^{\ot_k\,n+1}$ annihilates~$d$.
 The following result is a version of Nakayama lemma for comodules
and contramodules.

\begin{lem} \label{nakayama}
 Let\/ $\D$ be a conilpotent noncounital coalgebra.  Then \par
\textup{(a)} for any nonzero left\/ $\D$\+comodule\/ $\M$, the coaction
map\/ $\nu\:\M\rarrow\D\ot_k\M$ is \emph{not} injective; \par
\textup{(b)} for any nonzero left\/ $\D$\+contramodule\/ $\fP$,
the contraaction map\/ $\pi\:\Hom_k(\D,\fP)\allowbreak\rarrow\fP$ is
\emph{not} surjective.
\end{lem}

\begin{proof}
 Part~(a) can be found in~\cite[Lemma~2.1(a)]{Phff}.
 For part~(b), which is more complicated, see~\cite[Lemma~A.2.1]{Psemi};
and for a discussion of generalizations, \cite[Lemma~8.1]{Pksurv}
and~\cite[Lemmas~2.1 and~3.22]{Prev}.
\end{proof}

 The next lemma concerns infinite families of coalgebras.
 As such, it has no direct analogue for abstract algebras or rings
(but only for \emph{topological} rings with their topological products).

\begin{lem} \label{direct-sum-of-coalgebras}
 Let\/ $(\C_\xi)_{\xi\in\Xi}$ be a family of coalgebras and\/
$\C=\bigoplus_{\xi\in\Xi}\C_\xi$ be their direct sum, endowed with
its natural coalgebra structure. \par
\textup{(a)} The category of left\/ $\C$\+comodules is equivalent to
the Cartesian product of the categories of left\/ $\C_\xi$\+comodules
over the indices\/ $\xi\in\Xi$,
$$
 \C\Comodl\,\simeq\,\bigtimes\nolimits_{\xi\in\Xi}\,\C_\xi\Comodl.
$$
 Any left\/ $\C$\+comodule\/ $\M$ decomposes naturally into a direct
sum of left\/ $\C_\xi$\+comodules\/ $\M_\xi$ over\/ $\xi\in\Xi$,
$$
 \M\simeq\bigoplus\nolimits_{\xi\in\Xi}\M_\xi.
$$ \par
\textup{(b)} The category of left\/ $\C$\+contramodules is equivalent
to the Cartesian product of the categories of left\/
$\C_\xi$\+contramodules over the indices\/ $\xi\in\Xi$,
$$
 \C\Contra\,\simeq\,\bigtimes\nolimits_{\xi\in\Xi}\,\C_\xi\Contra.
$$
 Any left\/ $\C$\+contramodule\/ $\fP$ decomposes naturally into
a direct product of left\/ $\C_\xi$\+con\-tra\-mod\-ules\/ $\fP^\xi$
over\/ $\xi\in\Xi$,
$$
 \fP\simeq\prod\nolimits_{\xi\in\Xi}\fP^\xi.
$$
\end{lem}

\begin{proof}
 Part~(a) is easy: the subcomodule $\M_\xi\subset\M$ is recovered as
the full preimage of the direct summand $\C_\xi\ot_k\M\subset
\C\ot_k\M$ under the coaction map $\M\rarrow\C\ot_k\M$.
 Alternatively, part~(a) can be obtained as a particular case
of~\cite[Lemma~8.1(a)]{Pproperf}.
 Here the applicability of topological algebra language to comodules
over coalgebras over fields is established in~\cite[Section~2.1]{Swe}
(see the discussion in the beginning of
Section~\ref{change-of-scalars-subsecn} below,
cf.\ Remark~\ref{long-remark}).
 Part~(b) is~\cite[Lemma~A.2.2]{Psemi};
cf.~\cite[Lemma~8.1(b)]{Pproperf}, which is applicable in view
of~\cite[Section~2.3]{Prev}.
\end{proof}

\subsection{Change of scalars} \label{change-of-scalars-subsecn}
 We refer to~\cite[Section~3.1]{Pksurv} for a basic introductory
discussion of coalgebra homomorphisms, subcoalgebras and coideals.

 The content of this section can be viewed as a particular case
of~\cite[Section~2.9]{Pproperf}, but then the applicability of
the topological algebra exposition in~\cite{Pproperf} to comodules
and contramodules over coalgebras over a field needs to be
established beforehand.
 One needs to know that a left $\C$\+comodule is the same thing as
a discrete left module over the topological ring $\C^*$, and a left
$\C$\+contramodule is the same thing as a left contramodule over
the topological ring $\C^*$ with its natural
pro-finite-dimensional/linearly compact topology.
 In this connection, the relevant references
are~\cite[Section~2.1]{Swe} for comodules
and~\cite[Section~2.3]{Prev} for contramodules.
 The exposition below avoids any references to topological algebra.

 We start with some basic observations about ring homomorphisms,
stated in a slightly unusual notation for the purposes of
subsequent dualization to coalgebras.
 Let $f\:C\rarrow B$ be a ring homomorphism.
 Then, via the restriction of scalars, every left $B$\+module
acquires a left $C$\+module structure, and every right $B$\+module
acquires a right $C$\+module structure.
 So we have the exact, faithful \emph{restriction-of-scalars}
functor $f_*\:B\Modl\rarrow C\Modl$.
 This construction applies both to the unital and nonunital ring
homomorphisms.

 The functor~$f_*$ has adjoints on both sides.
 The left adjoint functor $f^*\:C\Modl\rarrow B\Modl$, defined by
the rule $f^*(N)=B\ot_CN$, is called the \emph{extension of scalars}.
 The right adjoint functor $f^!\:C\Modl\rarrow B\Modl$, given
by the rule $f^!(N)=\Hom_C(B,N)$, is called the \emph{coextension
of scalars}.
 These formulas presume $f$~to be a unital ring homomorphism.

 Now let us assume that the ring homomorphism $f\:C\rarrow B$ is
surjective with the kernel ideal $D\subset C$.
 Then the functor of restriction of scalars $f_*\:B\Modl\rarrow
C\Modl$ is exact and fully faithful.
 Furthermore, the ideal $D\subset C$ can be viewed as a nonunital
ring; then the inclusion map $D\rarrow C$ is a homomorphism of
nonunital rings.
 The restriction of scalars with respect to the inclusion map
$D\rarrow C$ endows every $C$\+module with a $D$\+module structure.

 For a surjective ring homomorphism $f\:C\rarrow B$,
the functors of (co)extension of scalars $f^*$ and~$f^!$ can be
computed as follows.
 For any left $C$\+module $N$, one has $f^*(N)=N/DN$.
 Here $N/DN$ is the unique maximal quotient $C$\+module of $N$
whose $C$\+module structure arises from a $B$\+module structure
via the restriction of scalars.
 For any left $C$\+module $N$, denote by ${}_DN\subset N$
the submodule of all elements annihilated by $D$ in~$M$.
 Then one has $f^!(N)={}_DN$.
 The submodule ${}_DN\subset N$ is the unique maximal $C$\+submodule
of $M$ whose $C$\+module structure arises from a $B$\+module
structure via the restriction of scalars.

 Passing from rings to coalgebras, the above picture of a functor
with adjoints on both sides splits in two halves.
 Let $g\:\B\rarrow\C$ be a homomorphism of coalgebras over
a field~$k$.
 Then any left $\B$\+comodule $\M$ acquires a natural structure of
left $\C$\+comodule, provided by the left $\C$\+coaction map
$$
 \M\lrarrow\B\ot_k\M\lrarrow\C\ot_k\M
$$
constructed as the composition of the left $\B$\+coaction map
with the map induced by the coalgebra homomorphism~$g$.
 Hence the exact, faithful functor of \emph{corestriction
of scalars} $g_\diamond\:\B\Comodl\rarrow\C\Comodl$.

 Dual-analogously, any left $\B$\+contramodule $\fP$ acquires
a natural structure of left $\C$\+contramodule, provided by
the left $\C$\+contraaction map
$$
 \Hom_k(\C,\fP)\lrarrow\Hom_k(\B,\fP)\lrarrow\fP
$$
constructed as the composition of the map induced by
the coalgebra homomorphism~$g$ with the left $\B$\+contraaction map.
 Hence the exact, faithful functor of \emph{contrarestriction
of scalars} $g_\sharp\:\B\Contra\rarrow\C\Contra$.
 These constructions of the functors $g_\diamond$ and~$g_\sharp$
apply both to the counital and noncounital coalgebra homomorphisms.

 The functor of corestriction of scalars~$g_\diamond$ has a right
adjoint functor $g^\diamond\:\C\Comodl\allowbreak\rarrow\B\Comodl$,
called the \emph{coextension of scalars}.
 Given a left $\C$\+comodule $\N$, the left $\B$\+comodule
$g^\diamond(\N)$ can be computed explicitly as the cotensor product
$g^\diamond(\N)=\B\oc_\C\N$ (see~\cite[Sections~2.1 and~4.8]{Pkoszul}
or~\cite[Sections~0.2.1, 1.2.1, and~7.1.2]{Psemi} for some details).

 Dual-analogously, the functor of contrarestriction of
scalars~$g_\sharp$ has a left adjoint functor $g^\sharp\:\C\Contra
\rarrow\B\Contra$, called the \emph{contraextension of scalars}.
 Given a left $\C$\+contramodule $\fQ$, the left
$\B$\+contramodlue $g^\sharp(\fP)$ can be computed explicitly as
$g^\sharp(\fQ)=\Cohom_\C(\B,\fQ)$
(see~\cite[Sections~2.2 and~4.8]{Pkoszul}
or~\cite[Sections~0.2.4, 3.2.1, and~7.1.2]{Psemi} for some details).
 These formulas for the functors $g^\diamond$ and~$g^\sharp$ presume
$g$~to be a counital coalgebra homomorphism.

 Now let us assume that the coalgebra homomorphism $g\:\B\rarrow\C$
is injective; so $\B$ can be viewed as a subcoalgebra in~$\C$.
 Then the functors of corestriction and contrarestriction of scalars
$g_\diamond\:\B\Comodl\rarrow\C\Comodl$ and
$g_\sharp\:\B\Contra\rarrow\C\Contra$ are exact and fully faithful.
 Furthermore, the quotient vector space $\D=\C/g(\B)$ acquires
a natural structure of noncounital coalgebra over~$k$; this coalgebra
structure is uniquely defined by the condition that the natural
surjective map $\C\rarrow\D$ must be a homomorphism of noncounital
coalgebras.
 The corestriction and contrarestriction of scalars with respect to
the noncounital coalgebra homomorphism $\C\rarrow\D$ endow every
$\C$\+comodule with a $\D$\+comodule structure, and every
$\C$\+contramodule with a $\D$\+contramodule structure.

 For an injective homomorphism of coalgebras $g\:\B\rarrow\C$,
the functors of coextension and contraextension of scalars
$g^\diamond$ and~$g^\sharp$ can be computed as follows.
 For any left $\C$\+comodule $\N$, denote by ${}_\D\N\subset\N$
the kernel of the left $\D$\+coaction map $\upsilon_\N\:\N
\rarrow\D\ot_k\N$.
 Then one has $g^\diamond(\N)={}_\D\N$.
 Here the vector subspace ${}_\D\N\subset\N$ is the unique maximal
$\C$\+subcomodule of $\N$ whose $\C$\+comodule structure arises
from a $\B$\+comodule structure via the corestriction of scalars.

 For any left $\C$\+contramodule $\fQ$, consider the left
$\D$\+contraaction map $\rho_\fQ\:
\Hom_k(\D,\allowbreak\fQ)\rarrow\fQ$.
 Then one has $g^\sharp(\fP)=\coker(\rho_\fQ)$.
 Here $\coker(\rho_\fQ)=\fQ/\im(\rho_\fQ)$ is the unique maximal
quotient $\C$\+contramodule of $\fQ$ whose $\C$\+contramodule
structure arises from a $\B$\+contramodule structure via
the contrarestriction of scalars.

\Section{Locally Finite Categories, Coalgebras, and Comodules}

 Let $k$~be a field.
 A \emph{$k$\+linear category} $\sE$ is a category enriched in
$k$\+vector spaces.
 So, for any two objects $x$, $y\in\sE$, the $k$\+vector space of
morphisms $\sE_{y,x}=\Hom_\sE(x,y)$ is given, together with
the multiplication/composition maps
$$
 \sE_{z,y}\ot_k\sE_{y,x}\rarrow\sE_{z,x},
 \text{ \ or equivalently, \ }
 \Hom_\sE(y,z)\ot_k\Hom_\sE(x,y)\rarrow\Hom_\sE(x,z)
$$
for all $x$, $y$, $z\in\sE$, and unit/identity elements
$\id_x\in\sE(x,x)$ for all $x\in\sE$.
 The composition maps must be $k$\+linear, and the usual associativity
and unitality axioms are imposed.
 For simplicity of notation, let us assume that there are no zero
objects in~$\sE$: so $\id_x\ne0$ for every $x\in\sE$.

 A \emph{left\/ $\sE$\+module} $M$ is a covariant $k$\+linear functor
$\sE\rarrow k\Vect$.
 So, for every object $x\in\sE$, a $k$\+vector space $M_x=M(x)$ is
given, together with the action maps
$$
 \sE_{y,x}\ot_k M_x\lrarrow M_y,
 \quad\text{or equivalently,}\quad
 \Hom_\sE(x,y)\ot_k M(x)\lrarrow M(y)
$$
for all $x$, $y\in\sE$.
 The action maps must be $k$\+linear, and the usual associativity
and unitality axioms must be satisfied.

 Similarly, a \emph{right\/ $\sE$\+module} $N$ is a contravariant
$k$\+linear functor $\sE^\sop\rarrow k\Vect$.
 So, for every object $x\in\sE$, a $k$\+vector space $N_x=N(x)$ is
given, together with the action maps
$$
 N_y\ot_k\sE_{y,x}\lrarrow N_x,
 \quad\text{or equivalently,}\quad
 \Hom_\sE(x,y)\ot_k N(y)\lrarrow N(x)
$$
for all $x$, $y\in\sE$.

\begin{defn} \label{morphism-preorder}
 Given a small $k$\+linear category $\sE$, we define a preorder
relation on the set of all objects of $\sE$ by the rule
\begin{itemize}
\item $x\preceq y$ if and only if there exists an integer $n\ge0$ and
objects $x=z_0$, $z_1$,~\dots, $z_{n-1}$, $z_n=y$ in $\sE$ such that
$\Hom_\sE(z_{i-1},z_i)\ne0$ for all $1\le i\le n$.
\end{itemize}
 In other words, we write $x\preceq y$ if there exists a finite chain
of composable nonzero morphisms in $\sE$ going from~$x$ to~$y$.

 Furthermore, let us write $x\prec y$ whenever $x\preceq y$
and $y\not\preceq x$.
 Let us also write $x\sim y$ whenever $x\preceq y$ and $y\preceq x$.
 
 The following piece of terminology is convenient.
 Let us say that a morphism $f\in\Hom_\sE(x,y)$ is \emph{short} if
either $f=0$, or $x\sim y$; and $f$~is \emph{long} if either $f=0$,
or $x\prec y$.
 Then short morphisms form a $k$\+linear subcategory in $\sE$, while
long morphisms form a two-sided ideal of morphisms in~$\sE$.
 Any nonzero morphism in $\sE$ is either short or long, but not both
(while the zero morphisms are both short and long).
\end{defn}

\begin{defn} \label{locally-finite-category}
 We will say that a small $k$\+linear category $\sE$ is \emph{locally
finite} if two conditions are satisfied:
\begin{enumerate}
\renewcommand{\theenumi}{\roman{enumi}}
\item finite-dimensional vector spaces of morphisms: for any two
objects $x$ and $y\in\sE$, the $k$\+vector space $\Hom_\sE(x,y)$ is
finite-dimensional;
\item interval finiteness of the preorder on objects: for any
two objects $x$ and $y\in\sE$, the set of objects $\{z\in\sE\mid
x\preceq z\preceq y\}$ is finite.
\end{enumerate}
\end{defn}

 Notice that our definition of a locally finite $k$\+linear category
is left-right symmetric: the passage to the opposite $k$\+linear
category $\sE^\sop$ with $\Hom_{\sE^\sop}(y,x)=\Hom_\sE(x,y)$ does not
affect it.
 So the results about left $\sE$\+modules in this section will be
equally applicable to right $\sE$\+modules.

\begin{constr} \label{coalgebra-construction}
 To every locally finite $k$\+linear category $\sE$, we assign
the related coalgebra $\C_\sE$ over~$k$ by the following rules.
 As a $k$\+vector space, we put
$$
 \C_\sE=\bigoplus\nolimits_{x,y\in\sE}\Hom_\sE(x,y)^*,
$$
where $W^*=\Hom_k(W,k)$ is the notation for the dual vector space
to a $k$\+vector space~$W$.
 Put $\C_\sE^{x,y}=\Hom_\sE(x,y)^*=\sE_{y,x}^*$; so
$\C_\sE=\bigoplus_{x,y\in\sE}\C_\sE^{x,y}$.
 The counit map $\epsilon\:\C_\sE\rarrow k$ is defined by the rules
that the restriction of~$\epsilon$ to $\C_\sE^{x,y}$ vanishes
for $x\ne y$, while the the restriction of~$\epsilon$ to $\C_\sE^{x,x}$
is the map
$$
 \epsilon_x\:\C_\sE^{x,x}=\Hom_\sE(x,x)^*\lrarrow k
$$
provided by the pairing with the identity morphism
$\id_x\in\Hom_\sE(x,x)$.
 Finally, the comultiplication map $\mu\:\C_\sE\rarrow\C_\sE\ot_k\C_\sE$
is defined by the rule that the restriction of~$\mu$ to $\C_\sE^{x,y}$
is the map
$$
 \mu_{x,y}\:\C_\sE^{x,y}\lrarrow
 \bigoplus\nolimits_{x\preceq z\preceq y}
 \C_\sE^{x,z}\ot_k\C_\sE^{z,y}
$$
whose components $\C_\sE^{x,y}\rarrow\C_\sE^{x,z}\ot_k\C_\sE^{z,y}$
are the dual maps to the composition/multiplication maps
$\sE_{y,z}\ot_k\sE_{z,x}\rarrow\sE_{y,x}$.

 Furthermore, consider the cosemisimple coalgebra
$\C^\id_\sE=\bigoplus_{x\in\sE}k$, with the coalgebra structure
of the direct sum of the one-dimensional coalgebras~$k$ over~$k$
(taken over all the objects $x\in\sE$).
 The rule $\sE\longmapsto\C_\sE$ assigns the coalgebra $\C^\id_\sE$
to the discrete $k$\+linear category $\sE^\id$ with the same set of
objects as $\sE$ (i.~e., the subcategory $\sE^\id\subset\sE$ whose
morphisms are all the scalar multiples of the identity morphisms
in~$\sE$).
 Then the inclusion of $k$\+linear categories $\sE^\id\rarrow\sE$
induces a surjective morphism of coalgebras $\C_\sE\rarrow\C^\id_\sE$.
 The latter map vanishes on the components $\C_\sE^{x,y}\subset\C_\sE$
with $x\ne y$, and is given by the maps~$\epsilon_x$ on the
components $\C_\sE^{x,x}\subset\C_\sE$.
\end{constr}

\begin{constr} \label{from-comodule-construction}
 Let $\sE$ be a locally finite $k$\+linear category.
 Then the comodule inclusion functor
$$
 \Upsilon_\sE\:\C_\sE\Comodl\lrarrow\sE\Modl
$$
is constructed as follows.
 Suppose given a left $\C_\sE$\+comodule~$\M$.
 Then the corestriction of scalars with respect to the morphism of
coalgebras $\C_\sE\rarrow\C_\sE^\id$ (see
Section~\ref{change-of-scalars-subsecn}) defines
a left $\C_\sE^\id$\+comodule structure on~$\M$.
 By Lemma~\ref{direct-sum-of-coalgebras}(a), the latter structure
is equivalent to the datum of a direct sum decomposition
$\M=\bigoplus_{x\in\sE}\M_x$.
 The $k$\+vector space $\M_x$ is the image of the idempotent linear
operator $\epsilon_x\:\M\rarrow\M$ provided by the action in $\M$
of the idempotent linear function $\C_\sE\twoheadrightarrow\C_\sE^{x,x}
\overset{\epsilon_x}\rarrow k$ on~$\C_\sE$.

 With respect to this direct sum decomposition of $\M$ and the direct
sum decomposition of $\C_\sE$ described in
Construction~\ref{coalgebra-construction}, the coaction map
$$
 \nu\:\M\lrarrow\C_\sE\ot_k\M
$$
has components
$$
 \nu_{x,y}\:\M_x\lrarrow\C_\sE^{x,y}\ot_k\M_y,
 \qquad y\in\sE,
$$
while all the other components vanish.
 More precisely, the map~$\nu$ is the direct sum of maps
\begin{equation} \label{nu-x-formula}
 \nu_x\:\M_x\lrarrow\bigoplus\nolimits_{y\in\sE}\C_\sE^{x,y}\ot_k\M_y,
\end{equation}
and the maps~$\nu_{x,y}$ are the compositions
$\M_x\overset{\nu_x}\rarrow\bigoplus\nolimits_{z\in\sE}
\C_\sE^{x,z}\ot_k\M_z\twoheadrightarrow\C_\sE^{x,y}\ot_k\M_y$.

 Now we put $M(x)=\M_x$ for all objects $x\in\sE$, and define
the action maps
$$
 \Hom_\sE(x,y)\ot_k M(x)\lrarrow M(y)
$$
as the $k$\+linear maps corresponding to~$\nu_{x,y}$ in view of
the isomorphisms $\Hom_\sE(x,y)\simeq(\C_\sE^{x,y})^*$ for
finite-dimensional $k$\+vector spaces $\C_\sE^{x,y}$.
 To finish the construction, it remains to set $\Upsilon_\sE(\M)=M$.
\end{constr}

\begin{prop} \label{comodule-fully-faithful}
 For any locally finite $k$\+linear category\/ $\sE$, the comodule
inclusion functor\/ $\Upsilon_\sE\:\C_\sE\Comodl\rarrow\sE\Modl$
is exact and fully faithful.
\end{prop}

\begin{proof}
 The exactness and faithfulness are obvious from the construction.
 Let us briefly explain why the functor is full.
 Essentially, the point is that the coaction maps $\nu_{x,y}\:\M_x
\rarrow\C_\sE^{x,y}\ot_k\M_y$ are uniquely determined by the action
maps $\Hom_\sE(x,y)\ot_k M(x)\rarrow M(y)$.
 Furthermore, the compatibility with the former is equivalent to
the compatibility with the latter, for any pair of left
$\C_\sE$\+comodules $\LL$ and $\M$ and a given $\C_\sE^\id$\+comodule
map $\LL\rarrow\M$.
 These properties are clear from the construction as well.
 It is illuminating to observe that a $k$\+linear map
$\nu_x\:\M_x\rarrow\bigoplus_{y\in\sE}\C_\sE^{x,y}\ot_k\M_y$
is determined by its components $\nu_{x,y}\:\M_x\rarrow
\C_\sE^{x,y}\ot_k\M_y$, essentially because the direct sum
$\bigoplus_{y\in\sE}\C_\sE^{x,y}\ot_k\M_y$ is a vector subspace
of the product $\prod_{y\in\sE}\C_\sE^{x,y}\ot_k\M_y$.
\end{proof}

\begin{defn} \label{locally-finite-modules-and-comodules}
 Let $\sE$ be a locally finite $k$\+linear category.
 A left $\sE$\+module $M$ is said to be \emph{locally finite} if,
viewed as a functor $\sE\rarrow k\Vect$, it takes values in
the category of finite-dimensional vector spaces $k\vect\subset
k\Vect$, that is $M\:\sE\rarrow k\vect$.
 In other words, $M$ is locally finite if and only if
the $k$\+vector space $M(x)$ is finite-dimensional for all $x\in\sE$.
 Locally finite right $\sE$\+modules are defined similarly.

 A left $\C_\sE$\+comodule $\M$ is said to be \emph{locally finite}
if the corresponding left $\sE$\+module $M=\Upsilon_\sE(\M)$ is
locally finite.
 In other words, $\M$ is locally finite if and only if the vector
space $\M_x$ is finite-dimensional for every $x\in\sE$.
 Locally finite right $\C_\sE$\+comodules are defined similarly.
\end{defn}

\begin{thm} \label{locally-finite-comodules-described}
 Let\/ $\sE$ be a locally finite $k$\+linear category and
$M\:\sE\rarrow k\vect$ be a locally finite left\/ $\sE$\+module.
 Then $M$ arises from some (locally finite) left\/ $\C_\sE$\+comodule
via the comodule inclusion functor\/ $\Upsilon_\sE$ if and only if,
for every fixed object $x\in\sE$, the set of all objects $y\in\sE$ with
nonzero action map
$$
 \Hom_\sE(x,y)\ot_k M(x)\lrarrow M(y)
$$
is finite.
\end{thm}

\begin{proof}
 ``Only if'': let $\M$ be a locally finite left $\C_\sE$\+comodule and
$x\in\sE$ be an object.
 Then, for every vector $v\in\M_x$, the tensor $\nu(v)\in\C_\sE\ot_k\M$
is a finite sum of decomposable tensors.
 Hence it must belong to the direct sum $\bigoplus_{y\in Y_v}
\C_\sE^{x,y}\ot_k\M_y\subset\C_\sE\ot_k\M$, where $Y_v\subset\sE$ is
some finite set of objects.
 Since the vector space $\M_x$ is finite-dimensional, it follows that
there exists a finite subset $Y_x\subset\sE$ such that the subspace
$\nu(\M_x)\subset\C_\sE\ot_k\M$ is contained in the direct sum
$\bigoplus_{y\in Y_x}\C_\sE^{x,y}\ot_k\M_y\subset\C_\sE\ot_k\M$.
 Now Construction~\ref{from-comodule-construction} tells that, for
the related left $\sE$\+module $M=\Upsilon_\sE(\M)$, the action map
$\Hom_\sE(x,y)\ot_k M(x)\lrarrow M(y)$ vanishes for all $y\notin Y_x$.

 ``If'': let $M$ be a (not necessarily locally finite) left
$\sE$\+module such that for every object $x\in\sE$ there is a finite
subset of objects $Y_x\subset\sE$ with the property that the action map
$\Hom_\sE(x,y)\ot_k M(x)\lrarrow M(y)$ vanishes for all $y\notin Y_x$.
 Put $\M_x=M(x)$ and $\M=\bigoplus_{x\in\sE}\M_x$, and define
the coaction map $\nu\:\M\rarrow\C_\sE\ot_k\M$ by the rule that its
restriction to $\M_x$ is the map
$$
 \nu_x\:\M_x\lrarrow\bigoplus\nolimits_{y\in Y_x}\C_\sE^{x,y}\ot_k\M_y
$$
as in~\eqref{nu-x-formula}, with the components $\nu_{x,y}\:\M_x
\rarrow\C_\sE^{x,y}\ot_k\M_y$ corresponding to the action maps
$\Hom_\sE(x,y)\ot_k M(x)\lrarrow M(y)$ in view of the isomorphisms
$\Hom_\sE(x,y)\simeq(\C_\sE^{x,y})^*$ for finite-dimensional
vector spaces~$\C_\sE^{x,y}$.
 This defines a left $\C_\sE$\+comodule structure on $\M$ such that
$\Upsilon_\sE(\M)=M$.
\end{proof}

\Section{Full-and-Faithfulness of Contramodule Forgetful Functor}
\label{contramodule-full-and-faithfulness-secn}

 The following construction is the contramodule version of
Construction~\ref{from-comodule-construction}.

\begin{constr} \label{from-contramodule-construction}
 Let $\sE$ be a locally finite $k$\+linear category.
 Then the contramodule forgetful functor
$$
 \Theta_\sE\:\C_\sE\Contra\lrarrow\sE\Modl
$$
is defined as follows.
 Suppose given a left $\C_\sE$\+contramodule~$\fP$.
 Then the contrarestriction of scalars with respect to the morphism of
coalgebras $\C_\sE\rarrow\C_\sE^\id$ (see
Construction~\ref{coalgebra-construction} and
Section~\ref{change-of-scalars-subsecn}) defines
a left $\C_\sE^\id$\+contramodule structure on~$\fP$.
 By Lemma~\ref{direct-sum-of-coalgebras}(b), the latter structure is
equivalent to the datum of a direct product decomposition
$\fP=\prod_{x\in\sE}\fP^x$.
 The $k$\+vector space $\fP^x$ is the image of the idempotent linear
operator $\epsilon_x\:\fP\rarrow\fP$ provided by the action in $\fP$
of the idempotent linear function $\C_\sE\twoheadrightarrow\C_\sE^{x,x}
\overset{\epsilon_x}\rarrow k$ on $\C_\sE$.

 With respect to this direct product decomposition of $\fP$ and
the direct sum decomposition of $\C_\sE$ described in
Construction~\ref{coalgebra-construction}, the contraaction map
$$
 \pi\:\Hom_k(\C_\sE,\fP)\lrarrow\fP
$$
has components
\begin{equation} \label{pi-y-formula}
 \pi^y\:\prod\nolimits_{x\in\sE}\Hom_k(\C_\sE^{x,y},\fP^x)
 \lrarrow\fP^y, \qquad y\in\sE,
\end{equation}
while all the other components vanish.
 Specifically, this means that the composition
$$
 \prod\nolimits_{x,y,z}^{(x,y)\ne(z,w)}\Hom_k(\C_\sE^{x,y},\fP^z)
 \,\rightarrowtail\Hom_k(\C_\sE,\fP)\overset\pi\rarrow\fP
 \,\twoheadrightarrow\,\fP^w
$$
vanishes for all $w\in\sE$; or in other words, the map~$\pi$
is the direct product of the maps~$\pi^y$.

 Now we put $P(x)=\fP^x$ for all objects $x\in\sE$, and define
the action maps
$$
 \Hom_\sE(x,y)\ot_k P(x)\lrarrow P(y)
$$
as the compositions
$$
 \Hom_\sE(x,y)\ot_k P(x)\,\simeq\,\Hom_k(\C_\sE^{x,y},P(x))
 \,\rightarrowtail\prod\nolimits_{z\in\sE}
 \Hom_k(\C_\sE^{z,y},P(z))\overset{\pi^y}\rarrow P(y),
$$
where the first isomorphism comes from the isomorphism
$\Hom_\sE(x,y)\simeq(\C_\sE^{x,y})^*$ for a finite-dimensional
$k$\+vector space~$\C_\sE^{x,y}$, while the middle morphism is
the natural direct summand inclusion.
 To finish the construction, it remains to set $\Theta_\sE(\fP)=P$.
\end{constr}

\begin{rem} \label{long-remark}
 For any coalgebra $\C$ over a field~$k$, the dual $k$\+vector space
$\C^*$ has a natural structure of a (topological) algebra over~$k$.
 The construction of the fully faithful comodule inclusion functor
$\Upsilon_\C\:\C\Comodl\rarrow\C^*\Modl$ goes back, at least,
to~\cite[Section~2.1]{Swe}; while the construction of
the contramodule forgetful functor $\Theta_\C\:\C\Contra\rarrow
\C^*\Modl$ was mentioned in~\cite[Section~A.1.2]{Psemi}.
 Our notation comes from the paper~\cite{Phff}, where the action of
the exact functors $\Upsilon_\C$ and $\Theta_\C$ on the Ext spaces
in the respective categories is discussed.
 In fact, the functor $\Upsilon_\C$ identifies the category of
$\C$\+comodules $\C\Comodl$ with the full subcategory in $\C^*\Modl$
consisting of all the \emph{discrete} modules over the topological
ring $\C^*$ (called ``rational'' in~\cite{Swe,Mon}).

 Let us explain the connection between the functors $\Upsilon_\sE$
and $\Theta_\sE$ defined in
Constructions~\ref{from-comodule-construction}
and~\ref{from-contramodule-construction}, on the one hand,
and the functors $\Upsilon_\C$ and $\Theta_\C$, on the other hand.
 For any preadditive category $\sE$, the category of left
$\sE$\+modules $\sE\Modl$ is equivalent to the category $R_\sE\Modl$
of locally unital left modules over the locally unital ring (ring
with enough idempotents) $R_\sE=\bigoplus_{x,y\in\sE}\Hom_\sE(x,y)$.
 In the situation at hand with a locally finite $k$\+linear category
$\sE$, the locally unital $k$\+algebra $R_\sE$ is a dense subalgebra
in the unital topological algebra $\C_\sE^*$ (with its natural
pro-finite-dimensional/linearly compact topology).

 The functor $\Upsilon_\sE\:\C_\sE\Comodl\rarrow\sE\Modl$ is
the composition of the functor $\Upsilon_{\C_\sE}\:\C_\sE\Comodl
\rarrow\C_\sE^*\Modl$ with the functor of restriction of scalars
$\C_\sE^*\Modl\rarrow R_\sE\Modl$.
 Similarly, the functor $\Theta_\sE\:\C_\sE\Contra\rarrow\sE\Modl$
is the composition of the functor $\Theta_{\C_\sE}\:\C_\sE\Contra
\rarrow\C_\sE^*\Modl$ with the functor of restriction of scalars
$\C_\sE^*\Modl\rarrow R_\sE\Modl$.
 The observations that the $\C_\sE$\+comodules are the same things as
the discrete $\C_\sE^*$\+modules and the subalgebra $R_\sE$ is dense
in $\C_\sE^*$ explain the fact that the functor $\Upsilon_\sE$ is
fully faithful, as per Proposition~\ref{comodule-fully-faithful}.
{\hbadness=1450\par}

 A \emph{coaugmentation}~$\gamma$ of a coalgebra $\C$ over a field~$k$
is a counital coalgebra homomorphism $\gamma\:k\rarrow\C$.
 A coaugmented coalgebra $(\C,\gamma)$ is said to be \emph{conilpotent}
if the noncounital quotient coalgebra $\D=\C/\gamma(k)$ is conilpotent
in the sense of the definition in Section~\ref{preliminaries-secn}.
 One can see that any conilpotent coalgebra $\C$ has a unique
coaugmentation~\cite[Section~3.3]{Pksurv}, \cite[Section~2]{Phff}.

 A conilpotent coalgebra $\C$ is said to be \emph{finitely cogenerated}
if the space $\Ext^1_\C(k,k)$, computed in the category of left
$\C$\+comodules $\C\Comodl$, with the $\C$\+comodule structure on~$k$
defined in terms of the coaugmentation~$\gamma$, is finite-dimensional.
 According to~\cite[Theorem~2.1]{Psm}, the contramodule forgetful
functor $\Theta_\C$ is fully faithful for any finitely cogenerated
conilpotent coalgebra~$\C$.
 Conversely, \cite[Example~7.2]{Phff} shows that, for any infinitely
cogenerated conilpotent coalgebra~$\C$, the functor $\Theta_\C$ is
not full.

 In fact, \cite[Theorem~2.1]{Psm} tells more.
 Let $\C$ be a finitely cogenerated conilpotent coalgebra and
$R\subset\C^*$ be a dense subring in the natural
(pro-finite-dimensional/linearly compact) topology on~$\C^*$.
 Then the composition of the functor $\Theta_\C\:\C\Contra\rarrow
\C^*\Modl$ with the functor of restriction of scalars
$\C^*\Modl\rarrow R\Modl$ is a fully faithful functor
$\C\Contra\rarrow R\Modl$. {\hbadness=1050\par}

 The coalgebra $\C_\sE$ is usually \emph{not} conilpotent, of course.
 The aim of this section is to obtain an analogue
of~\cite[Theorem~2.1]{Psm} for the coalgebras~$\C_\sE$.
\end{rem}

\begin{ex} \label{contra-ff-counterex}
 It is obvious from the construction that the functor $\Theta_\sE$ is
exact and faithful for any locally finite $k$\+linear category~$\sE$.
 The following counterexample shows, however, that the functor
$\Theta_\sE$ \emph{need not} be fully faithful for a locally finite
$k$\+linear category $\sE$ in general.

 Let $\sE$ be the $k$\+linear category whose set of objects is
the set of all negative integers $\boZ_{<0}$.
 For any integer $n\ge1$, the $k$\+vector space $\Hom_\sE(-n,-n)$
is one-dimensional and spanned by the identity morphism~$\id_{-n}$.
 For any integer $n\ge2$, the $k$\+vector space $\Hom_\sE(-n,-1)$
is one-dimensional and spanned by a morphism $f_{-n,-1}\in
\Hom_\sE(-n,-1)$.
 There are no other nonzero morphisms in~$\sE$.

 The category of left $\C_\sE$\+comodules agrees with the category
of left $\sE$\+modules; in other words, the functor $\Upsilon_\sE$ is
a category equivalence (cf.\
Proposition~\ref{upper-finite-comodule-equivalence} below).
 The datum of a left $\C_\sE$\+comodule is equivalent to the datum
of a family of $k$\+vector spaces $\M_{-n}$, \,$n\ge1$, together with
arbitrary $k$\+linear maps
$$
 f_{-n,-1}\:\M_{-n}\lrarrow\M_{-1}, \qquad n\ge2.
$$
 The $\C_\sE$\+comodule $\M$ is recovered as
$\M=\bigoplus_{n\ge1}\M_{-n}$.

 The category of right $\C_\sE$\+comodules does \emph{not} agree with
the category of right $\sE$\+modules; so the functor
$\Upsilon_{\sE^\sop}$ is \emph{not} an equivalence.
 The datum of a right $\C_\sE$\+comodule is equivalent to the datum
of a family of $k$\+vector spaces $\N_{-n}$, \,$n\ge1$, together with
an arbitrary $k$\+linear map
$$
 G\:\N_{-1}\lrarrow\bigoplus\nolimits_{n\ge2}\N_{-n}.
$$
 The $\C_\sE$\+comodule $\N$ is recovered as
$\N=\bigoplus_{n\ge1}\N_{-n}$.
 In particular, if $\N_{-1}$ is finite-dimensional, then the datum
of a map $G$ is a equivalent to that of a family of linear maps
$g_{-n,-1}\:\N_{-1}\rarrow\N_{-n}$, \,$n\ge2$, such that \emph{only
finitely many of the maps $g_{-n,-1}$ are nonzero}
(cf.\ Theorem~\ref{locally-finite-comodules-described}).
 On the other hand, the datum of a right $\sE$\+module is equivalent
to that of a family of $k$\+vector spaces $N(-n)$, \,$n\ge1$,
together with arbitrary $k$\+linear maps
$$
 g_{-n,-1}\:N(-1)\lrarrow N(-n), \qquad n\ge2.
$$

 The datum of a left $\C_\sE$\+contramodule $\fP$ is equivalent to
the datum of a family of $k$\+vector spaces $\fP^{-n}$, \,$n\ge1$,
together with an arbitrary $k$\+linear map
$$
 F\:\prod\nolimits_{n\ge2}\fP^{-n}\lrarrow\fP^{-1}.
$$
 The $\C_\sE$\+contramodule $\fP$ is recovered as
$\fP=\prod_{n\ge1}\fP^{-n}$.
 The functor $\Theta_\sE$ assigns to $\fP$ the family of $k$\+vector
spaces $P(-n)=\fP^{-n}$ endowed with the family of maps $f_{-n,-1}\:
P(-n)\rarrow P(-1)$ obtained as restrictions of the map~$F$.
 Essentially, this means restricting the map $F$ to the vector subspace
$\bigoplus_{n\ge2}\fP^{-n}\subset\prod_{n\ge2}\fP^{-n}$.

 For example, let $\fP^{-n}=k$ be one-dimensional $k$\+vector spaces
for all $n\ge1$.
 Let $F'\:\prod_{n\ge2}\fP^{-n}\rarrow\fP^{-1}$ be a nonzero linear map
annihilating the subspace $\bigoplus_{n\ge2}\fP^{-n}\subset
\prod_{n\ge2}\fP^{-n}$, and let $F''\:\prod_{n\ge2}\fP^{-n}\rarrow
\fP^{-1}$ be the zero map.
 Then the $\C_\sE$\+contramodules $\fP'$ and $\fP''$ defined by
the maps $F'$ and $F''$ are \emph{not} isomorphic to each other, but
their images under the functor $\Theta_\sE$ \emph{are} isomorphic as
left $\sE$\+modules.
 Thus the functor $\Theta_\sE$ is \emph{not} fully faithful.
\end{ex}

 We start with a discussion of a trivial particular case when
it is easy to show that the functor $\Theta_\sE$ is not only fully
faithful, but even a category equivalence.
 Then we will proceed to introduce a much more general sufficient
condition on a locally finite $k$\+linear category $\sE$
guaranteeing that the functor $\Theta_\sE$ is fully faithful.

\begin{rem} \label{coalgebras-on-short-and-long-morphisms-remark}
 Let $\sE$ be a locally finite $k$\+linear category.
 Denote by $\sE^\sim\subset\sE$ the $k$\+linear subcategory of
\emph{short morphisms} in~$\sE$ (in the sense of
Definition~\ref{morphism-preorder}).
 This means that, given two objects $x$, $y\in\sE$, we put
$\Hom_{\sE^\sim}(x,y)=\Hom_\sE(x,y)$ if $x\sim y$ and
$\Hom_{\sE^\sim}(x,y)=0$ otherwise.
 So we have the $k$\+linear inclusion functor $\sE^\sim\rarrow\sE$.

 We are more interested, however, in the natural functor in
the opposite direction.
 The functor $\sE\rarrow\sE^\sim$ acts by the identity map on
the objects and by surjective maps on the morphisms, and represents
$\sE^\sim$ as the quotient category of $\sE$ by the two-sided ideal
of \emph{long morphisms}.
 More precisely, the functor $\sE\rarrow\sE^\sim$ takes every short
morphism to itself and every long morphism to zero.

 Consider the vector subspace
$$
 \C_\sE^\sim=\bigoplus\nolimits_{x,y\in\sE}^{x\sim y}\C_\sE^{x,y}
$$
in the coalgebra~$\C_\sE$.
 One can easily see from the definition of the preorder~$\preceq$
on the set of objects of $\sE$ that $\C_\sE^\sim$ is a subcoalgebra
in~$\C_\sE$.
 The coalgebra $\C_\sE^\sim$ arises via
Construction~\ref{coalgebra-construction} applied to the locally finite
$k$\+linear category $\sE^\sim$, that is
$$
 \C_\sE^\sim=\C_{\sE^\sim}.
$$
 The inclusion of coalgebras $\C_\sE^\sim\rarrow\C_\sE$ is induced
by the surjective $k$\+linear functor $\sE\rarrow\sE^\sim$ (while
the inclusion of categories $\sE^\sim\rarrow\sE$ induces a surjective
morphism of coalgebras $\C_\sE\rarrow\C_\sE^\sim$, which we will
never use).

 The quotient space of any coalgebra by a subcoalgebra has a natural
structure of noncounital coalgebra (see
Section~\ref{change-of-scalars-subsecn},
cf.\ the discussion in~\cite[Section~3.1]{Pksurv}).
 So the quotient vector space
$$
 \D_\sE^\prec=\C_\sE/\C_\sE^\sim=
 \bigoplus\nolimits_{x,y\in\sE}^{x\prec y}\C_\sE^{x,y}
$$
is naturally a noncounital coalgebra over~$k$.
 The natural surjective morphism (direct summand projection)
$\C_\sE\rarrow\D_\sE^\prec$ is a morphism of noncounital coalgebras.

 By the contrarestriction of scalars (see
Section~\ref{change-of-scalars-subsecn},
cf.~\cite[Section~8.4]{Pksurv}), any left $\C_\sE$\+contramodule
acquires a left $\D_\sE^\prec$\+contramodule structure, while
any left $\C_\sE^\sim$\+con\-tra\-mod\-ule structure on a vector space
$\fW$ gives rise to a left $\C_\sE$\+contramodule structure on~$\fW$.
 The notation in this remark compares with the one in
Section~\ref{change-of-scalars-subsecn} by the rules
$\B=\C_\sE^\sim$, \ $\C=\C_\sE$, and $\D=\D_\sE^\prec$.

 Our proofs below are based on the following important observation:
 The condition of interval finiteness of the preorder~$\preceq$
from Definition~\ref{locally-finite-category} implies that
the noncounital coalgebra $\D_\sE^\prec$ is conilpotent.
 So Lemma~\ref{nakayama} is applicable.
\end{rem}

\begin{lem} \label{coalgebra-on-short-morphisms-lemma}
 For any locally finite $k$\+linear category\/ $\sE$, the contramodule
forgetful functor\/ $\Theta_{\sE^\sim}\:\C_\sE^\sim\Contra\rarrow
\sE^\sim\Modl$ is an equivalence of categories.
\end{lem}

\begin{proof}
 Let $\Xi$ denote the set of all equivalence classes of objects
of the category $\sE$ with respect to the equivalence relation~$\sim$.
 For every $\xi\in\Xi$, denote by $\sE^\xi\subset\sE^\sim\subset\sE$
the full subcategory on the set of all objects $x\in\xi$.
 Then the $k$\+linear category $\sE^\sim$ is the (preadditive)
disjoint union of the $k$\+linear categories $\sE^\xi$ over all
$\xi\in\Xi$.
 Similarly, the coalgebra $\C_\sE^\sim$ is the direct sum of
the coalgebras $\C_\sE^\xi=\C_{\sE^\xi}$,
$$
 \C_\sE^\sim=\bigoplus\nolimits_{\xi\in\Xi}\C_\sE^\xi.
$$
 Furthermore, it is clear from
Definition~\ref{locally-finite-category} that the coalgebra
$\C_\sE^\xi$ is finite-dimensional for every $\xi\in\Xi$.
 The finite-dimensional, unital algebra $R_{\sE^\xi}=
\bigoplus_{x,y\in\sE^\xi}\Hom_{\sE^\xi}(x,y)$
is naturally isomorphic to the dual algebra $\bigl(\C_\sE^\xi\bigr)^*$
to the coalgebra~$\C_\sE^\xi$.
 Hence the functor $\Theta_{\sE^\xi}\:\C_\sE^\xi\Contra\rarrow
\sE^\xi\Modl$ is a category equivalence.

 It remains to refer to Lemma~\ref{direct-sum-of-coalgebras}(b)
to the effect that any $\C_\sE^\sim$\+contramodule decomposes
naturally as an infinite product of $\C_\sE^\xi$\+contramodules
over all $\xi\in\Xi$.
 So the category $\C_\sE^\sim\Contra$ is the Cartesian product of
the categories $\C_\sE^\xi\Contra$.
 Similarly, the category $\sE^\sim\Contra$ is clearly
the Cartesian product of the categories $\sE^\xi\Contra$.
 The functor $\Theta_{\sE^\sim}$ is the Cartesian product of
the category equivalences $\Theta_{\sE^\xi}$, hence also
a category equivalence.

 Alternatively, one could observe that the $k$\+linear category
$\sE^\sim$ is lower (and upper) finite in the sense of
Definition~\ref{upper-lower-finite} for any locally finite
$k$\+linear category~$\sE$.
 So it suffices to refer to
Proposition~\ref{lower-finite-contramodule-equivalence} below.
\end{proof}

 The following definition spells out a property of locally finite
$k$\+linear categories $\sE$ implying that the functor
$\Theta_\sE$ is fully faithful.

\begin{defn} \label{lower-strict-locally-finite}
 Let $\sE$ be a locally finite $k$\+linear category.
 We will say that $\sE$ is \emph{lower strictly locally finite} if
the following condition is satisfied: For any object $y\in\sE$
there exists a finite set of objects $X_y\subset\sE$ such that
$x\prec y$ (in the sense of Definition~\ref{morphism-preorder})
for all $x\in X_y$ and, for every morphism $f\in\Hom_\sE(z,y)$
with $z\prec y$ there exist two finite families of morphisms $g_i\:
z\rarrow x_i$ and $h_i\:x_i\rarrow y$, \ $1\le i\le n$, \ $n\ge0$,
\ $x_i\in X_y$, such that $f=\sum\nolimits_{i=1}^n h_ig_i$.
 In other words, this means that the multiplication map
$\bigoplus_{x\in X_y}\Hom_\sE(x,y)\ot_k\Hom_\sE(z,x)\lrarrow
\Hom_\sE(z,y)$ should be surjective for all $z\in\sE$ such that
$z\prec y$.

 Dually, a locally finite $k$\+linear category $\sE$ is said
to be \emph{upper strictly locally finite} if the opposite
category $\sE^\sop$ is lower strictly locally finite.
\end{defn}

 For example, the $k$\+linear category $\sE$ from
Example~\ref{contra-ff-counterex} is \emph{not} lower strictly
locally finite (take $y=-1$).

\begin{thm} \label{main-full-and-faithfulness-theorem}
 Let\/ $\sE$ be a lower strictly locally finite $k$\+linear category.
 Then the contramodule forgetful functor\/ $\Theta_\sE\:\C_\sE\Contra
\rarrow\sE\Modl$ is fully faithful.
\end{thm}

\begin{proof}
 Let $\fP$ and $\fQ$ be two left $\C_\sE$\+contramodules, and let
$\phi\:\Theta_\sE(\fP)\rarrow\Theta_\sE(\fQ)$ be a morphism of
the related left $\sE$\+modules.
 Following Construction~\ref{from-contramodule-construction}, both
$\fP$ and $\fQ$ decompose naturally as infinite products
$\fP=\prod_{x\in\sE}\fP^x$ and $\fQ=\prod_{x\in\sE}\fQ^x$.
 The related left $\sE$\+modules $P$ and $Q$ are defined so that
$P(x)=\fP^x$ and $Q(x)=\fQ^x$ for all $x\in\sE$.
 The $\sE$\+module morphism~$\phi$ is a collection of maps
$\phi(x)\:P(x)\rarrow Q(x)$.
 Consider the product of maps $\prod_{x\in\sE}\phi(x)\:\fP\rarrow\fQ$,
and denote it also by~$\phi$.
 We need to show that~$\phi$ is a $\C_\sE$\+contramodule morphism.

 The contraaction map $\pi\:\Hom_k(\C_\sE,\fP)\rarrow\fP$ is
a surjective morphism of left $\C_\sE$\+contramodules (where
the vector space $\Hom_k(\C_\sE,\fP)$ is endowed with the free
left $\C_\sE$\+contramodule structure defined in
Section~\ref{preliminaries-secn}).
 In order to show that $\phi\:\fP\rarrow\fQ$ is a $\C_\sE$\+contramodule
morphism, is suffices to check that $\phi\circ\pi\:\Hom_k(\C_\sE,\fP)
\rarrow\fQ$ is a $\C_\sE$\+contramodule morphism.
 This reduces the original problem to the case when $\fP$ is a free
left $\C_\sE$\+contramodule.
 Without loss of generality, we will therefore assume that
$\fP=\Hom_k(\C_\sE,V)$ is a free $\C_\sE$\+contramodule generated
by some $k$\+vector space~$V$.

 The counit map $\C_\sE\rarrow k$ induces a natural inclusion
$V\rarrow\Hom_k(\C_\sE,V)$ representing $V$ as a space of free
generators of $\Hom_k(\C_\sE,V)$.
 The composition $V\rarrow\Hom_k(\C_\sE,V)\overset\phi\rarrow\fQ$
defines a $k$\+linear map $V\rarrow\fQ$, which can be uniquely
extended to a left $\C_\sE$\+contramodule map $\psi\:\Hom_k(\C_\sE,V)
\rarrow\fQ$.
 By construction, the maps~$\psi$ and $\phi\:\Hom_k(\C_\sE,V)
\rarrow\fQ$ agree in restriction to the vector subspace of free
$\C_\sE$\+contramodule generators $V\subset\Hom_k(\C_\sE,V)$;
we need to show that they agree on the whole $\C_\sE$\+contramodule
$\Hom_k(\C_\sE,V)$.

 Put $\chi=\phi-\psi$.
 Then $\chi\:\Hom_k(\C_\sE,V)\rarrow\fQ$ is a $k$\+linear map
about which we know that it corresponds, via the construction from
the first paragraph of this proof, to a left $\sE$\+module
morphism $\Theta_\sE(\Hom_k(\C_\sE,V))\rarrow\Theta_\sE(\fQ)$.
 We also know that the restriction of~$\chi$ to the vector subspace
of free $\C_\sE$\+contramodule generators $V\subset\Hom_k(\C_\sE,V)$
vanishes.
 We need to show that $\chi$~is a $\C_\sE$\+contramodule morphism,
or equivalently, that $\chi=0$.
 Replacing if needed the $\C_\sE$\+contramodule $\fQ$ by its
subcontramodule generated by the image of~$\chi$, we can assume
without loss of generality that no proper $\C_\sE$\+subcontramodule
of $\fQ$ contains the image of~$\chi$.
 Then our task is to prove that $\fQ=0$.
 It still suffices to show that~$\chi$ is a $\C_\sE$\+contramodule
morphism.

 Now we are using the notation and discussion of
Section~\ref{change-of-scalars-subsecn} and
Remark~\ref{coalgebras-on-short-and-long-morphisms-remark}.
 Assume for the sake of contradiction that $\fQ\ne0$.
 Consider the composition $\rho_\fQ\:\Hom_k(\D_\sE^\prec,\fQ)\rarrow\fQ$
of the injective map $\Hom_k(\D_\sE^\prec,\fQ)\rarrow\Hom_k(\C_\sE,\fQ)$
induced by the surjective map $\C_\sE\rarrow\D_\sE^\prec$ with
the $\C_\sE$\+contraaction map $\pi_\fQ\:\Hom_k(\C_\sE,\fQ)\rarrow\fQ$.
 By the contramodule Nakayama lemma (Lemma~\ref{nakayama}(b)),
the $\D_\sE^\prec$\+contraaction map~$\rho_\fQ$ \emph{cannot}
be surjective.
 The image of the map~$\rho_\fQ$ is a $\C_\sE$\+subcontramodule
in~$\fQ$; in fact, it is the maximal $\C_\sE$\+subcontramodule
$\fS\subset\fQ$ with the property that the $\C_\sE$\+contramodule
structure on the quotient contramodule $\fQ/\fS$ arises from
a $\C_\sE^\sim$\+contramodule structure. {\hbadness=1525\par}

 So $\fS=\im(\rho_\fQ)$ is a proper $\C_\sE$\+subcontramodule
of~$\fQ$; and we are already assuming that no proper
$\C_\sE$\+subcontramodule of $\fQ$ contains the image of~$\chi$.
 Replacing the contramodule $\fQ$ by the contramodule $\fQ/\fS$ and
the map $\chi\:\Hom_k(\C_\sE,V)\rarrow\fQ$ by its composition with
the natural surjection $\fQ\rarrow\fQ/\fS$, we can assume without loss
of generality that the left $\C_\sE$\+contramodule structure on $\fQ$
arises from a left $\C_\sE^\sim$\+contramodule structure.
 Informally, one can say that ``the contraaction of long morphisms
in $\fQ$ vanishes'' (after we've killed it by taking the quotient
by~$\fS$).

 At this point, let us return to the notation $\fP=\Hom_k(\C_\sE,V)$.
 Just as the maps $\phi$ and~$\psi$, and all such maps throughout
our argument, the map~$\chi$ decomposes as the product
$$
 \chi=\prod\nolimits_{x\in\sE}\chi^x\:
 \prod\nolimits_{x\in\sE}\fP^x\lrarrow\prod\nolimits_{x\in\sE}\fQ^x.
$$
 We have $P=\Theta_\sE(\fP)$ and $Q=\Theta_\sE(\fQ)$, where
$P(x)=\fP^x$ and $Q(x)=\fQ^x$ for all $x\in\sE$.
 The collection of maps $\chi(x)=\chi^x\:P(x)\rarrow Q(x)$ is
a left $\sE$\+module morphism $\chi\:P\rarrow Q$.

 Our next aim is to show that the contraaction of long morphisms
in $\fP$ is annihilated by~$\chi$.
 Specifically, the claim is that the composition
\begin{equation} \label{long-contraaction-composed-with-chi}
 \Hom_k(\D_\sE^\prec,\fP)\overset{\rho_\fP}\lrarrow
 \fP\overset\chi\lrarrow\fQ
\end{equation}
vanishes.
 Here $\rho_\fP$~is the composition $\Hom_k(\D_\sE^\prec,\fP)
\,\rightarrowtail\,\Hom_k(\C_\sE,\fP)\overset{\pi_\fP}\rarrow\fP$,
similarly to the notation in
Section~\ref{change-of-scalars-subsecn} and above.
 The vanishing of~\eqref{long-contraaction-composed-with-chi} would be
obvious if we knew $\chi$~to be a $\C_\sE$\+contramodule morphism;
the point is that we only know that $\chi$~is an $\sE$\+module morphism.

 With respect to the direct product decomposition over $x\in\sE$,
the map~\eqref{long-contraaction-composed-with-chi} consists of
the components
\begin{equation} \label{long-contraaction-composed-with-chi-y-component}
 \Hom_k(\D_\sE^\prec,\fP)^y\,=\,
 \prod\nolimits_{x\in\sE}^{x\prec y}
 \Hom_k(\C_\sE^{x,y},\fP^x)\overset{\rho_\fP^y}\lrarrow
 \fP^y\overset{\chi^y}\lrarrow\fQ^y, \qquad y\in\sE.
\end{equation}
 In order to prove that
the composition~\eqref{long-contraaction-composed-with-chi} vanishes,
it suffices to check that
the composition~\eqref{long-contraaction-composed-with-chi-y-component}
vanishes for every $y\in\sE$.

 So let us fix an object $y\in\sE$.
 We still have not used the assumption that $\sE$ is
lower strictly locally finite so far in this proof.
 Now we are using it.
 Consider the finite set of objects $X_y\subset\sE$ corresponding to
the object~$y$ as per Definition~\ref{lower-strict-locally-finite}.
 Put
$$
 U_y=\bigoplus\nolimits_{x\in X_y}\C_\sE^{x,y}.
$$
 So $U_y$ is a finite-dimensional $k$\+vector space.
 Denote by $q_y\:\C_\sE\rarrow U_y$ the direct summand projection.

\begin{lem} \label{U-cogenerators-lemma}
 Let\/ $\sE$ be a lower strictly locally finite $k$\+linear category
and $y\in\sE$ be an object.
 Let\/ $\fP$ be a left\/ $\C_\sE$\+contramodule,
$\rho_\fP\:\Hom_k(\D_\sE^\prec,\fP)\rarrow\fP$ be its\/
$\D_\sE^\prec$\+contraaction map, and 
\begin{equation} \label{long-contraaction-map}
 \prod\nolimits_{z\in\sE}^{z\prec y}
 \Hom_k(\C_\sE^{z,y},\fP^z)\overset{\rho_\fP^y}\lrarrow\fP^y
\end{equation}
be its $y$\+component.
 Then the image of
the contraaction map~\eqref{long-contraaction-map} is contained in
(equivalently, is equal to) the image of the composition
\begin{equation} \label{X_y-contraaction-map}
 \prod\nolimits_{x\in X_y}
 \Hom_k(\C_\sE^{x,y},\fP^x)\lrarrow
 \prod\nolimits_{x\in\sE}^{x\prec y}
 \Hom_k(\C_\sE^{x,y},\fP^x)\overset{\rho_\fP^y}\lrarrow\fP^y.
\end{equation}
 Here the leftmost map in~\eqref{X_y-contraaction-map} is
the direct summand inclusion induced by the direct summand
projection $q_y\:\C_\sE\rarrow U_y$, in the notation above.
\end{lem}

\begin{proof}
 Fix an object $z\in\sE$ such that $z\prec y$, and consider
the composition
\begin{equation} \label{U-cogeneration-map}
 \C_\sE^{z,y}\lrarrow
 \bigoplus\nolimits_{x\in\sE}
 \C_\sE^{z,x}\ot_k\C_\sE^{x,y}\lrarrow
 \bigoplus\nolimits_{x\in X_y}
 \C_\sE^{z,x}\ot_k\C_\sE^{x,y}.
\end{equation}
 Here the leftmost map is the $(z,y)$\+component of
the comultiplication in $\C_\sE$, while the rightmost map is induced by
the direct summand projection~$q_y$.

 The key observation is that the condition in
Definition~\ref{lower-strict-locally-finite} implies injectivity
of the composition~\eqref{U-cogeneration-map}.
 Indeed, the domain of the composition~\eqref{U-cogeneration-map} is
the dual vector space to $\Hom_\sE(z,y)$.
 The condition in Definition~\ref{lower-strict-locally-finite} tells
that the composition/multiplication map in the category~$\sE$,
$$
 \bigoplus\nolimits_{x\in X_y}\Hom_\sE(x,y)\ot_k\Hom_\sE(z,x)
 \lrarrow\Hom_\sE(x,z),
$$
which is $k$\+vector space dual to 
the composition in~\eqref{U-cogeneration-map}, is surjective.
 Hence the composition in~\eqref{U-cogeneration-map} is injective.

 Now injectivity of the map~\eqref{U-cogeneration-map} implies
surjectivity of the map obtained by applying the contravariant
functor $\Hom_k({-},\fP^z)$ to~\eqref{U-cogeneration-map}.
 Taking the direct product over $\{z\mid z\prec y\}$,
we conclude that the map
\begin{equation} \label{Hom-from-U-cogeneration-map}
 \prod\nolimits_{z\in\sE,\,x\in X_y}^{z\prec y}
 \Hom_k(\C_\sE^{x,y},\Hom_k(\C_\sE^{z,x},\fP^z))\lrarrow
 \prod\nolimits_{z\in\sE}^{z\prec y}
 \Hom_k(\C_\sE^{z,y},\fP^z)
\end{equation}
induced by the comultiplication in $\C_\sE$ is surjective.

 Finally, it remains to apply the contraassociativity axiom.
 The contraassociativity axiom for the $\C_\sE$\+contramodule $\fP$
describes the composition of
the map~\eqref{Hom-from-U-cogeneration-map} with the contraaction
map~\eqref{long-contraaction-map} as the composition of two maps
induced by the contraaction map~$\pi_\fP$.
 It follows the image of the map~\eqref{long-contraaction-map} is
contained in the image of the map~\eqref{X_y-contraaction-map},
as desired.
\end{proof}

 Informally speaking, Lemma~\ref{U-cogenerators-lemma} tells that
\emph{the image of the contraaction of the long morphisms in\/ $\fP$,
projected to\/ $\fP^y$, is contained in the image of the contraaction
of the morphisms from $U_y$, projected to\/~$\fP^y$}.
 But $X_y$ is only a finite set of objects, and $U_y$ is only
a finite-dimensional vector space.
 For this reason, the ``contraaction of the morphisms from~$U_y$''
in the $\C_\sE$\+contramodule $\fP$ is a part of the action of
the category $\sE$ in the $\sE$\+module~$P$.
 And we know that the map $\chi\:P\rarrow Q$ preserves
the $\sE$\+module structure.
 We also know that the contraaction of long morphisms in $\fQ$ vanishes;
so in particular the action of long morphisms in $Q$ vanishes.
 And all the morphisms from the objects $x\in X_y$ to
the object~$y$ are long.
 These arguments explain why
the composition~\eqref{long-contraaction-composed-with-chi-y-component}
has to vanish.

 More precisely, the composition in~\eqref{X_y-contraaction-map} is
isomorphic to the $\sE$\+action map
\begin{equation} \label{X_y-action-map}
 \bigoplus\nolimits_{x\in X_y}\Hom_\sE(x,y)\ot_k P(x)\lrarrow P(y).
\end{equation}
 Since the image of~\eqref{long-contraaction-map} is contained in
(actually, coincides with) the image of~\eqref{X_y-contraaction-map},
and since $\chi\:P\rarrow Q$ is a left $\sE$\+module morphism, it
follows that the image of
the composition~\eqref{long-contraaction-composed-with-chi-y-component}
is contained in the image of the composition
\begin{equation} \label{X_y-action-map-composed-with-chi(y)}
 \bigoplus\nolimits_{x\in X_y}\Hom_\sE(x,y)\ot_k P(x)\lrarrow P(y)
 \overset{\chi(y)}\lrarrow Q(y),
\end{equation}
which is contained in the image of the $\sE$\+action map
\begin{equation} \label{X_y-action-in-Q}
 \bigoplus\nolimits_{x\in X_y}\Hom_\sE(x,y)\ot_k Q(x)\lrarrow Q(y).
\end{equation}
 Finally, the image of~\eqref{X_y-action-in-Q} is contained in
the image of the $\D_\sE^\prec$\+contraaction map
\begin{equation} \label{long-contraaction-map-in-Q}
 \prod\nolimits_{x\in\sE}^{x\prec y}
 \Hom_k(\C_\sE^{x,y},\fQ^x)\overset{\rho_\fQ^y}\lrarrow\fQ^y,
\end{equation}
which we know to be zero by the preceding arguments in this proof.

 We have shown that
the composition~\eqref{long-contraaction-composed-with-chi-y-component}
vanishes, and consequently so does
the composition~\eqref{long-contraaction-composed-with-chi},
as claimed.
 Let $\fT\subset\fP$ be the image of the $\D_\sE^\prec$\+contraaction
map $\rho_\fP\:\Hom_k(\D_\sE^\prec,\fP)\rarrow\fP$.
 So $\fT$ is the maximal $\C_\sE$\+subcontramodule in $\fP$ with
the property that the $\C_\sE$\+contramodule structure on $\fP/\fT$
arises from a $\C_\sE^\sim$\+contramodule structure.
 We have shown that the $k$\+linear map $\chi\:\fP\rarrow\fQ$
factorizes as $\fP\rarrow\fP/\fT\rarrow\fQ$.
 By construction, $\fP\rarrow\fP/\fT$ is a $\C_\sE$\+contramodule
morphism.
 It remains to prove that $\fP/\fT\rarrow\fQ$ is
a $\C_\sE$\+contramodule morphism.

 Put $T=\Theta_\sE(\fT)$; so $T$ is a left $\sE$\+module.
 Applying the functor $\Theta_\sE$ to the surjective
$\C_\sE$\+contramodule morphism $\fP\rarrow\fP/\fT$, we obtain
a surjective $\sE$\+module morphism $P\rarrow P/T$.
 The $\sE$\+module morphism $\chi\:P\rarrow Q$ factorizes as
$P\rarrow P/T\rarrow Q$, where $P/T\rarrow Q$ is
an $\sE$\+module morphism.
 Now the $\sE$\+module structure on $P/T=\Theta_\sE(\fP/\fT)$
arises from a $\sE^\sim$\+module structure (in other words,
the action of the ideal of long morphisms vanishes in~$P/T$),
since the $\C_\sE$\+contramodule structure on $\fP/\fT$ arises
from a $\C_\sE^\sim$\+contramodule structure.
 Similarly, the $\sE$\+module structure on $Q$ arises from
a $\sE^\sim$\+module structure (i.~e., once again, the action
of the ideal of long morphisms vanishes in~$Q$).
 So $P/T\rarrow Q$ is a morphism of left $\sE^\sim$\+modules.

 At last, by Lemma~\ref{coalgebra-on-short-morphisms-lemma},
the functor $\Theta_{\sE^\sim}$ is an equivalence of categories;
in particular, it is fully faithful.
 Since $P/T\rarrow Q$ is a morphism of $\sE^\sim$\+modules,
it follows that $\fP/\fT\rarrow\fQ$ is a morphism of
$\C_\sE^\sim$\+contramodules.
 Thus $\fP/\fT\rarrow\fQ$ is a morphism of $\C_\sE$\+contramodules
as well.
 We can conclude that the composition of two $\C_\sE$\+contramodule
morphisms $\fP\rarrow\fP/\fT\rarrow\fQ$ is also
a $\C_\sE$\+contramodule morphism, finishing the proof of the theorem.
\end{proof}

\Section{Locally Finite Contramodules}
\label{locally-finite-contramodules-secn}

 The following definition is the contramodule version
of Definition~\ref{locally-finite-modules-and-comodules}.

\begin{defn} \label{locally-finite-contramodules-defn}
 Let $\sE$ be a locally finite $k$\+linear category.
 A left $\C_\sE$\+contramodule $\fP$ is said to be \emph{locally
finite} if the corresponding left $\sE$\+module $P=\Theta_\sE(\fP)$
(defined in Construction~\ref{from-contramodule-construction})
is locally finite.
 In other words, $\fP$ is locally finite if and only if the vector
space $\fP^x$ is finite-dimensional for every object $x\in\sE$.
 Locally finite right $\C_\sE$\+contramodules are defined similarly.
\end{defn}

 The aim of this section is to prove the following contramodule
version of Theorem~\ref{locally-finite-comodules-described}.

\begin{thm} \label{locally-finite-contramodules-described}
 Let\/ $\sE$ be a lower strictly locally finite $k$\+linear category
and $P\:\sE\rarrow k\vect$ be a locally finite left\/ $\sE$\+module.
 Then $P$ arises from some (locally finite) left\/
$\C_\sE$\+contramodule via the contramodule forgetful functor\/
$\Theta_\sE$ if and only if, for every fixed object $y\in\sE$,
the set of all objects $x\in\sE$ with nonzero action map
$$
 \Hom_\sE(x,y)\ot_k P(x)\lrarrow P(y)
$$
is finite.
\end{thm}

 Notice that the functor $\Theta_\sE$ is fully faithful in
the assumptions of
Theorem~\ref{locally-finite-contramodules-described},
by Theorem~\ref{main-full-and-faithfulness-theorem}.
 So the $\C_\sE$\+contramodule $\fP$ such that $\Theta_\sE(\fP)
\simeq P$ is unique if it exists.
 Let us also emphasize that the lower strict local finiteness
assumption on the $k$\+linear category $\sE$ \emph{cannot} be simply
dropped or replaced by the local finiteness in
Theorem~\ref{locally-finite-contramodules-described}.
 Indeed, \emph{every} left module $P$ over the $k$\+linear category
$\sE$ from Example~\ref{contra-ff-counterex} belongs to
the essential image of the functor $\Theta_\sE$, but $P$ need not
satisfy the condition of
Theorem~\ref{locally-finite-contramodules-described}
for the object $y=-1$.

\begin{proof}[Proof of the ``if'' implication in
Theorem~\ref{locally-finite-contramodules-described}]
 This implication requires \emph{neither} the lower strict local
finiteness assumption on $\sE$, \emph{nor} the local finiteness
assumption on~$P$.
 Let $P$ be a (not necessarily locally finite) left
$\sE$\+module such that for every object $y\in\sE$ there is
a finite subset of objects $Z_y\subset\sE$ with the property that
the action map $\Hom_\sE(z,y)\ot_k P(z)\rarrow P(y)$ vanishes
for all $z\in\sE$, \,$z\notin Z_y$.
 Put $\fP^x=P(x)$ and $\fP=\prod_{x\in\sE}\fP^x$, and define
the contraaction map $\pi\:\Hom_k(\C_\sE,\fP)\rarrow\fP$ by
the following rule.
 The map~$\pi$ consists of the components
$$
 \pi^y\:\prod\nolimits_{x\in\sE}\Hom_k(\C_\sE^{x,y},\fP^x)
 \lrarrow\fP^y, \qquad y\in\sE
$$
as in~\eqref{pi-y-formula}.
 The map~$\pi^y$ is the composition
\begin{multline} \label{contraaction-constructed}
 \prod\nolimits_{x\in\sE}\Hom_k(\C_\sE^{x,y},\fP^x)
 \lrarrow\prod\nolimits_{z\in Z_y}\Hom_k(\C_\sE^{z,y},\fP^z) \\
 \,\simeq\,\bigoplus\nolimits_{z\in Z_y}
 \Hom_\sE(z,y)\ot_k P(z)\lrarrow P(y).
\end{multline}
 Here the leftmost map in~\eqref{contraaction-constructed} is
the direct summand projection from the product of the components
indexed by all the objects of $\sE$ to the product of
the components indexed by the subset $Z_y\subset\sE$.
 The rightmost map in~\eqref{contraaction-constructed} is provided
by the action of $\sE$ in~$P$.
 The middle isomorphism holds since $\Hom_\sE(z,y)\simeq
\bigl(\C_\sE^{z,y}\bigr)^*$ is a finite-dimensional vector space
and the set $Z_y$ is finite.
 This defines a left $\C_\sE$\+contramodule structure on $\fP$
such that $\Theta_\sE(\fP)=P$.
\end{proof}

 Before proceeding to prove the ``only if'' implication, let us
introduce a further series of combinatorial definitions and
formulate some lemmas.

\begin{defn}
 Let $\sE$ be a locally finite $k$\+linear category and $x$, $y\in\sE$
be a pair of objects such that $x\preceq y$.
 The \emph{distance} $d(x,y)$ from $x$ to~$y$ is defined as
the supremum of the set of all integers $n\ge0$ for which there
exists a sequence of objects $z_1$,~\dots, $z_{n-1}\in\sE$ such that
$x\prec z_1\prec z_2\prec\dotsb\prec z_{n-1}\prec y$.
 In particular, $d(x,y)\ge1$ whenever $x\prec y$.
 We put $d(x,y)=0$ if $x\sim y$.

 The point of this definition is that $d(x,y)$ is a finite integer
for every pair of objects $x\preceq y$, because the set of all
objects $\{z\mid x\preceq z\preceq y\}$ is finite by
Definition~\ref{locally-finite-category}(ii).
\end{defn}

\begin{defn}
 Let $\sE$ be a locally finite $k$\+linear category and $y\in\sE$
be an object.
 Let $X\subset\sE$ be a finite subset in $\sE$ such that $x\prec y$
for all $x\in X$.
 We will say that $X$ is a \emph{frontier} for~$y$ if there is
a finite subset of objects $W\subset\sE$ such that, for for every
morphism $f\in\Hom_\sE(z,y)$ with $z\prec y$ and $z\notin W$,
there exist two finite families of morphisms $g_i\: z\rarrow x_i$
and $h_i\:x_i\rarrow y$, \ $1\le i\le n$, \ $n\ge0$,
\ $x_i\in X$, for which $f=\sum\nolimits_{i=1}^n h_ig_i$.
 The finite subset of objects $W\subset\sE$ will be called
the \emph{frontier exception} for a frontier $X$ of an object~$y$.
\end{defn}

 In particular, Definition~\ref{lower-strict-locally-finite} can be
restated by saying that a locally finite $k$\+linear category $\sE$
is lower strictly locally finite if and only every object $y\in\sE$
admits a frontier.
 To be more precise, Definition~\ref{lower-strict-locally-finite}
requires existence of a frontier with an empty frontier exception.
 An empty exception can be always achieved by expanding the frontier.
 For any frontier $X$ of an object~$y$ with a frontier exception $W$,
the union $X\cup(W\cap\{z\in\sE\mid z\prec y\})$ is a frontier
for~$y$ with an empty frontier exception.

\begin{defn}
 Let $\sE$ be a lower strictly locally finite $k$\+linear category.
 Suppose chosen a finite subset of objects $X_y\subset\sE$ for every
object $y\in\sE$ as per Definition~\ref{lower-strict-locally-finite}.
 Let us call $X_y$ the \emph{standard frontier} of~$y$.

 Put $X_y^{(1)}=X_y$.
 Proceeding by induction on $n\ge1$, set
$X_y^{(n+1)}=\bigcup_{z\in X_y^{(n)}} X_z$.
 So $X_y^{(n)}$ is a finite set of objects in $\sE$ for every
object $y\in\sE$.
 Let us call $X_y^{(n)}$ the \emph{degree~$n$ standard frontier}
of~$y$.
\end{defn}

\begin{lem} \label{degree-n-standard-frontier}
 Our terminology is consistent: for any object $y\in\sE$ and any
integer $n\ge1$, the degree~$n$ standard frontier $X_y^{(n)}$ is
a frontier for~$y$.
\end{lem}

\begin{proof}
 Quite generally, if $X$ is a frontier for~$y$ with a frontier
exception $W$, then $\bigcup_{z\in X}X_z$ is also a frontier for~$y$
with the frontier exception $W\cup X^\sim$, where $X^\sim$ is
the set of all objects of $\sE$ that are $\sim$\+equivalent to
some objects from~$X$.
 Since the $\sim$\+equivalence classes are finite by
Definition~\ref{lower-strict-locally-finite}(ii), the set $X^\sim$
is finite.
 Hence $X_y^{(n)}$ is a frontier for~$y$ with a frontier exception
$\bigcup_{i=1}^{n-1} X_y^{(n)\,\sim}$.
\end{proof}

 The following definition assigns a name to the class of left
$\sE$\+modules we are interested in.

\begin{defn} \label{contrafinite-module}
 Let\/ $\sE$ be a locally finite $k$\+linear category and $P$ be
a left\/ $\sE$\+module.
 We will say that $P$ is \emph{contrafinite} if it satisfies
the condition of
Theorem~\ref{locally-finite-contramodules-described}.
 To repeat, this means that, for every object $y\in\sE$, there
exists a finite subset of objects $A_{P,y}\subset\sE$ such that,
for all objects $x\in\sE$, \ $x\notin A_{P,y}$, the action map
$\Hom_\sE(x,y)\ot_k P(x)\rarrow P(y)$ vanishes.
\end{defn}

 It is clear that the class of all contrafinite left $\sE$\+modules
is closed under submodules, quotients, and finite direct sums.
 The next lemma plays a key role.

\begin{lem} \label{contrafinite-extension-closed}
 Let\/ $\sE$ be a lower strictly locally finite $k$\+linear category.
 Then the class of all contrafinite left\/ $\sE$\+modules is closed
under extensions in\/ $\sE\Modl$. 
\end{lem}

\begin{proof}
 Let $0\rarrow P\rarrow T\rarrow Q\rarrow0$ be a short exact sequence
of left $\sE$\+modules.
 Assume that the $\sE$\+modules $P$ and $Q$ are contrafinite.
 We have to prove that the $\sE$\+module $T$ is contrafinite as well.
 Given an object $y\in\sE$, we need to establish existence of a finite
subset of objects $A_{T,y}\subset\sE$.

 Consider the degree~$n$ standard frontier $X_y^{(n)}\subset\sE$,
and let $W_y^{(n)}\subset\sE$ be its frontier exception constructed
in the proof of Lemma~\ref{degree-n-standard-frontier}.
 For any object $x\in X_y^{(n)}$, the distance $d(x,y)$ is greater
than or equal to~$n$.
 As the distance between a pair of objects $u$ and~$v$
with $u\preceq v$ is always finite, we can choose $n$~large enough
so that \emph{no object of $A_{P,y}\cup A_{Q,y}$ belongs to
$X_y^{(n)}$}.
 Then, for any object $x\in X_y^{(n)}$, both the maps
$\Hom_\sE(x,y)\ot_k P(x)\rarrow P(y)$ and
$\Hom_\sE(x,y)\ot_k Q(x)\rarrow Q(y)$ vanish.
 It follows that for any object $z\prec y$ with $z\notin W_y^{(n)}$,
both the maps $\Hom_\sE(x,y)\ot_k P(x)\rarrow P(y)$ and
$\Hom_\sE(x,y)\ot_k Q(x)\rarrow Q(y)$ vanish; but we will not need
to use this fact.

 Now consider the degree~$m$ standard frontier $X_y^{(m)}\subset\sE$
with $m>n$.
 The set $B=\bigcup_{x\in X_y^{(n)}}(A_{P,x}\cup A_{Q,x})$ is finite.
 For the same reason as above, we can choose~$m$ large enough so that
\emph{no object of $B$ belongs to $X_y^{(m)}$}.
 Then, for any objects $x\in X_y^{(n)}$ and $z\in X_y^{(m)}$, both
the maps $\Hom_\sE(z,x)\ot_k P(z)\rarrow P(x)$ and
$\Hom_\sE(z,x)\ot_k Q(z)\rarrow Q(x)$ vanish.
 It follows that for any objects $x\in X_y^{(n)}$ and $u\prec x$
with $u\notin W_x^{(m-n)}$, both the maps
$\Hom_\sE(u,x)\ot_k P(u)\rarrow P(x)$ and
$\Hom_\sE(u,x)\ot_k Q(u)\rarrow Q(x)$ vanish; now this fact is
relevant for us.

 Put $A_{T,y}=\{y\}^\sim\cup W_y^{(m)}$, where $\{y\}^\sim\subset\sE$
is the finite set of all objects $\sim$\+equivalent to~$y$.
 Given an object $u\notin A_{T,y}$ and a nonzero morphism
$f\in\Hom_\sE(u,y)$, we have to show that the map
$T(f)\:T(u)\rarrow T(y)$ vanishes.
 Notice that we necessarily have $u\prec y$ and $u\notin W_y^{(n)}$.
 Since $X_y^{(n)}$ is a frontier for~$y$ with a frontier exception
$W_y^{(n)}$ by Lemma~\ref{degree-n-standard-frontier}, the map~$f$
can be represented as
$$
 f=\sum\nolimits_{i=1}^r h_ig_i,
$$
where $g_i\in\Hom_\sE(u,x_i)$ and $h_i\in\Hom_\sE(x_i,y)$, \
$1\le i\le r$, \ $r\ge1$, and $x_i\in X_y^{(n)}$.
 Without loss of generality we can assume that $g_i\ne0$ for
every $1\le i\le r$.

 For a given index~$i$, \ $1\le i\le r$, put $g=g_i$, \,$h=h_i$
and $x=x_i$.
 The existence of a nonzero morphism $g\:u\rarrow x$ and
the condition that $u\notin W_y^{(m)}\supset X_y^{(n)\,\sim}$
imply that $u\prec x$.
 Since $W_x^{(m-n)}\subset W_y^{(m)}$, we also have $u\notin
W_x^{(m-n)}$.
 Therefore, both the maps $P(g)\:P(u)\rarrow P(x)$ and
$Q(g)\:Q(u)\rarrow Q(x)$ vanish.
 So do both the maps $P(h)\:P(x)\rarrow P(y)$ and $Q(h)\:
Q(x)\rarrow Q(y)$, as we have seen in the second paragraph
of this proof.

 Finally, we can conclude that the map $T(hg)\:T(u)\rarrow T(y)$
vanishes, as the product of any two strictly uppertriangular
$2\times 2$ block matrices is zero.
 Thus $T(f)=\sum_{i=1}^r T(h_ig_i)=0$, as desired.
\end{proof}

 Now the following sequence of lemmas leads to the proof of
the remaining (main) implication
in Theorem~\ref{locally-finite-contramodules-described}.
 To begin with, here is another piece of auxiliary terminology.
 Given a left $\sE$\+module $P$, we will say that
an $\sE$\+submodule $Q\subset P$ is \emph{big} if the quotient
$\sE$\+module $P/Q$ is contrafinite.

\begin{lem} \label{finite-intersection-big}
 Let\/ $\sE$ be a locally finite $k$\+linear category and $P$ be
a left\/ $\sE$\+module.
 Let $Q'$ and $Q''\subset P$ be two big $\sE$\+submodules in~$P$.
 Then the intersection $Q'\cap Q''$ is also a big\/ $\sE$\+submodule
in~$P$.
\end{lem}

\begin{proof}
 The quotient module $P/(Q'\cap Q'')$ is a submodule of the direct
sum $P/Q'\oplus P/Q''$.
 So the assertion follows from the facts that the class of all
contrafinite left $\sE$\+modules is closed under submodules and
finite direct sums in $\sE\Modl$.
\end{proof}

\begin{lem} \label{infinite-intersection-in-locally-finite}
 Let\/ $\sE$ be a locally finite $k$\+linear category and $P$ be
a locally finite left\/ $\sE$\+module.
 Let $(Q_\xi\subset P)_{\xi\in\Xi}$ be a family of big\/
$\sE$\+submodules in~$P$.
 Then the intersection\/ $\bigcap_{\xi\in\Xi}Q_\xi$ is also a big\/
$\sE$\+submodule in~$P$.
\end{lem}

\begin{proof}
 Given a left $\sE$\+module $T$ and an object $y\in\sE$, let us say
that the $\sE$\+module $T$ is \emph{$y$\+contrafinite} if a finite
subset of objects $A_{T,y}\subset\sE$ satisfying the condition of
Definition~\ref{contrafinite-module} exists (for one particular
object~$y$).
 Given a left $\sE$\+module $P$, let us say that an $\sE$\+submodule
$K\subset P$ is \emph{$y$\+big} if the quotient $\sE$\+module $P/K$
is $y$\+contrafinite.
 The key observation is that, for a fixed $\sE$\+module $P$ and
a fixed object $y\in\sE$, the property of an $\sE$\+submodule
$K\subset P$ to be $y$\+big depends only on the vector subspace
$K(y)\subset P(y)$ (and \emph{not} on the vector subspaces $K(x)
\subset P(x)$ for all the other objects $x\ne y$ in~$\sE$).

 Returning to the situation at hand, for every given object $y\in\sE$,
the $k$\+vector space $P(y)$ is finite-dimensional by assumption.
 Hence there exists a finite subset of indices $\Upsilon_y\subset\Xi$
such that $\bigcap_{\xi\in\Xi}Q_\xi(y)=
\bigcap_{\upsilon\in\Upsilon_y}Q_\upsilon(y)$.
 Now the $\sE$\+submodule $\bigcap_{\upsilon\in\Upsilon_y}
Q_\upsilon\subset P$ is big, and in particular $y$\+big,
by Lemma~\ref{finite-intersection-big}.
 It follows that the $\sE$\+submodule
$\bigcap_{\xi\in\Xi}Q_\xi\subset P$ is $y$\+big.
 As this holds for all $y\in\sE$, we can conclude that
the $\sE$\+submodule $\bigcap_{\xi\in\Xi}Q_\xi$ is big in~$P$.
\end{proof}

\begin{lem} \label{proper-big-subcontramodule}
 Let\/ $\sE$ be a locally finite $k$\+linear category and\/ $\fP$ be
a nonzero left\/ $\C_\sE$\+contramodule.
 Then the left\/ $\sE$\+module $\Theta_\sE(\fP)$ has a proper big\/
$\sE$\+submodule.
\end{lem}

\begin{proof}
 The argument is based on the discussion in
Section~\ref{change-of-scalars-subsecn} and
Remark~\ref{coalgebras-on-short-and-long-morphisms-remark},
together with the contramodule Nakayama lemma
(Lemma~\ref{nakayama}(b)).
 Let $\rho_\fP\:\Hom_k(\D_\sE^\prec,\allowbreak\fP)\rarrow\fP$
denote the $\D_\sE^\prec$\+contraaction map.
 Following Section~\ref{change-of-scalars-subsecn},
the vector subspace $\fQ=\im(\rho_\fP)\subset\fP$ is
a $\C_\sE$\+subcontramodule.
 Furthermore, the $\C_\sE$\+contramodule structure on $\fP/\fQ$
arises from a $\C_\sE^\sim$\+contramodule structure.
 By Lemma~\ref{nakayama}(b), \,$\fQ$ is a proper subcontramodule
of~$\fP$.

 It remains to put $Q=\Theta_\sE(\fQ)$ and notice that $Q$ is
a proper $\sE$\+submodule of~$P$ such that the quotient
$\sE$\+module $P/Q$ arises from an $\sE^\sim$\+module.
 It is clear from the interval finiteness condition in
Definition~\ref{locally-finite-category} that any left
$\sE^\sim$\+module is contrafinite (also as an $\sE$\+module).
 So $P/Q$ is a nonzero contrafinite left $\sE$\+module and
$Q$ is a proper big $\sE$\+submodule in~$P$.
\end{proof}

\begin{proof}[Proof of the ``only if'' implication in
Theorem~\ref{locally-finite-contramodules-described}]
 Let $\fP$ be a locally finite left $\C_\sE$\+con\-tra\-mod\-ule.
 In the terminology of Definition~\ref{contrafinite-module},
we need to show that the left $\sE$\+module $P=\Theta_\sE(\fP)$
is contrafinite.

 By Lemma~\ref{infinite-intersection-in-locally-finite},
the intersection of all big $\sE$\+submodules in $P$ is
a big $\sE$\+submodule $K\subset P$.
 So $K$ is the unique minimal big $\sE$\+submodule in~$P$.
 The quotient $\sE$\+module $Q=P/K$ is locally finite and contrafinite.
 In order to prove that $P$ is contrafinite, we will show that $K=0$.

 By the ``if'' implication of the theorem (which we have proved already
in the beginning of this section), the left $\sE$\+module $Q$ arises
from some left $\sE$\+contramodule $\fQ$; so $Q=\Theta_\sE(\fQ)$.
 As the functor $\Theta_\sE$ is (exact and) fully faithful by
Theorem~\ref{main-full-and-faithfulness-theorem}, the surjective
$\sE$\+module morphism $P\rarrow Q$ arises from a surjective
$\C_\sE$\+contramodule morphism $\fP\rarrow\fQ$.
 Put $\fK=\ker(\fP\to\fQ)$.
 Then $K=\Theta_\sE(\fK)$; so the $\sE$\+submodule $K\subset P$
also arises from a $\C_\sE$\+subcontramodule $\fK\subset\fP$.

 Assume for the sake of contradiction that $K\ne0$.
 Then Lemma~\ref{proper-big-subcontramodule} tells that the left
$\sE$\+module $K$ has a proper big $\sE$\+submodule $L\subset K$.
 Now we have a short exact sequence of left $\sE$\+modules
$0\rarrow K/L\rarrow P/L\rarrow P/K=Q\rarrow0$.
 The $\sE$\+submodule $L$ is big in $K$, and the $\sE$\+submodule
$K$ is big in~$P$; so both the $\sE$\+modules $K/L$ and $Q$ are
contrafinite.
 By Lemma~\ref{contrafinite-extension-closed}, it follows that
the $\sE$\+module $P/L$ is contrafinite, too.
 So $L$ is a big $\sE$\+submodule in~$P$.

 But $L$ is a proper $\sE$\+submodule of $K$, and $K$ is the (unique)
minimal big $\sE$\+submodule in $P$ by construction.
 The contradiction proves that $K=0$ and the left $\sE$\+module $P$
is contrafinite.
\end{proof}

\begin{qst}
 Let $\sE$ be a lower strictly locally finite $k$\+linear category and
$\fP$ be a left $\C_\sE$\+contramodule.
 Fix an object $y\in\sE$ and assume that the $k$\+vector space $\fP^y$
(for one particular object~$y$) is finite-dimensional.
 Put $P=\Theta_\sE(\fP)$.
 Does it follow that there exists a finite set of objects $A_{P,y}
\subset\sE$ satisfying the condition of
Definition~\ref{contrafinite-module} (for the particular object~$y$)?
 In other words, is the left $\sE$\+module $P$ necessarily
$y$\+contrafinite in the sense of the terminology in the proof
of Lemma~\ref{infinite-intersection-in-locally-finite}\,?
 Our proof of the ``only if'' implication in
Theorem~\ref{locally-finite-contramodules-described} does not seem to
answer this question.
\end{qst}

\Section{Vector Space Duality}  \label{vector-space-dualization-secn}

 Let $\sE$ be a $k$\+linear category and $N\:\sE^\sop\rarrow k\Vect$
be a right $\sE$\+module.
 Then the rule $P(x)=N(x)^*=\Hom_k(N(x),k)$ defines a left
$\sE$\+module $P=N^*\:\sE\rarrow k\Vect$.

 On the other hand, let $\C$ be a coalgebra over~$k$ and $\N$ be
a right $\C$\+comodule.
 Then the construction from Section~\ref{preliminaries-secn} defines
a left $\C$\+contramodule structure on the $k$\+vector space
$\fP=\N^*=\Hom_k(\N,k)$.

 Now let $\sE$ be a locally finite $k$\+linear category, and let
$\C_\sE$ be the related coalgebra as per
Construction~\ref{coalgebra-construction}.
 Then we have a commutative square of exact, faithful functors
\begin{equation} \label{dualization-square}
\begin{gathered}
 \xymatrix{
  (\Comodr\C_\sE)^\sop \ar[rr]^-{\N\mapsto\N^*}
  \ar[d]_{(\Upsilon_{\sE^\sop})^\sop}
  && \C_\sE\Contra \ar[d]^{\Theta_\sE} \\
  (\Modr\sE)^\sop \ar[rr]^-{N\mapsto N^*}
  && \sE\Modl
 }
\end{gathered}
\end{equation}
 Here the functor $\Upsilon_{\sE^\sop}$ is the right (co)module
version of the functor $\Upsilon_\sE$ from
Construction~\ref{from-comodule-construction}, while the functor
$\Theta_\sE$ was defined in
Construction~\ref{from-contramodule-construction}.

 Notice that a $\C_\sE$\+comodule $\N$ is a direct sum of its
components, $\N=\bigoplus_{x\in\sE}\N_x$, while
a $\C_\sE$\+contramodule $\fP$ is a direct product of its
components, $\fP=\prod_{x\in\sE}\fP^x$.
 The diagram~\eqref{dualization-square} is commutative due to
the fact that the vector space dualization functor $\N\longmapsto\N^*$
takes the direct sums to the direct products.

 Let us introduce some further notation.
 Denote the full subcategories of locally finite $\sE$\+modules
(see Definition~\ref{locally-finite-modules-and-comodules}) by
$\sE\Modl_\lf\subset\sE\Modl$ and $\Modrlf\sE\subset\Modr\sE$.
 It is clear that the functor $N\longmapsto P=N^*$ restricts to
a anti-equivalence of categories
\begin{equation} \label{lf-modules-dualization-equivalence}
 (\Modrlf\sE)^\sop \xrightarrow[\,\,N\mapsto N^*\,]\simeq
 \sE\Modl_\lf.
\end{equation}
 This follows simply from the fact that the dual vector space functor
is an anti-equivalence on the category of finite-dimensional
vector spaces, $(k\vect)^\sop\simeq k\vect$.

 Furthermore, denote the full subcategories of locally finite
$\C_\sE$\+comodules by $\C_\sE\Comodl_\lf\subset\C_\sE\Comodl$
and $\Comodrlf\C_\sE\subset\Comodr\C_\sE$.
 Similarly, let the full subcategory of locally finite left
$\C_\sE$\+contramodules (see
Definition~\ref{locally-finite-contramodules-defn}) be denoted by
$\C_\sE\Contra_\lf\subset\C_\sE\Contra$.
 It is clear that the functor $\N\longmapsto\fP=\N^*$ restricts
to an exact, faithful functor
\begin{equation} \label{lf-co-contra-modules-dualization}
 (\Comodrlf\C_\sE)^\sop \xrightarrow{\,\,\N\mapsto\N^*\,}
 \C_\sE\Contra_\lf.
\end{equation}
 In fact, one can say more.

\begin{prop}
 For any locally finite $k$\+linear category\/ $\sE$, the contravariant
functor $(\Comodrlf\C_\sE)^\sop\rarrow\C_\sE\Contra_\lf$
\,\eqref{lf-co-contra-modules-dualization} is fully faithful.
\end{prop}

\begin{proof}
 For any coalgebra $\C$ over~$k$, a $k$\+vector space map
$f\:\M\rarrow\N$ between two right comodules $\M$, $\N$ over $\C$ is
a $\C$\+comodule morphism if and only if the dual map
$f^*\:\N^*\rarrow\M^*$ is a left $\C$\+contramodule morphism.
 Not all $\C$\+contramodule morphisms $g\:\N^*\rarrow\M^*$, viewed
as $k$\+vector space maps, arise from $k$\+vector space maps
$\M\rarrow\N$ in general.
 But in the situation at hand they do, because all morphisms
$g\:\fP\rarrow\fQ$ in $\C_\sE\Contra$ respect the decomposition into
the product of the components, $\fP=\prod_{x\in\sE}\fP^x$ and
$\fQ=\prod_{x\in\sE}\fQ^x$, and the components $\M_x$ and $\N_x$ of
all comodules $\M$, $\N\in\Comodrlf\C_\sE$ are finite-dimensional
by the definition.
\end{proof}

 Before stating a corollary
(of Theorems~\ref{locally-finite-comodules-described}
and~\ref{locally-finite-contramodules-described}), let us formulate
explicitly the right (co)module version of
Theorem~\ref{locally-finite-comodules-described}.

\begin{prop} \label{locally-finite-right-comodules}
 Let\/ $\sE$ be a locally finite $k$\+linear category and
$N\:\sE^\sop\rarrow k\vect$ be a locally finite right\/ $\sE$\+module.
 Then $N$ arises from some (locally finite) right\/ $\C_\sE$\+comodule
via the comodule inclusion functor\/ $\Upsilon_{\sE^\sop}$ if and
only if, for every fixed object $y\in\sE$, the set of all objects
$x\in\sE$ with nonzero action map
$$
 \Hom_\sE(x,y)\ot_k N(y)\lrarrow N(x)
$$
is finite.
\end{prop}

\begin{proof}
 Apply Theorem~\ref{locally-finite-comodules-described} to
the locally finite $k$\+linear category $\sE^\sop$ in place of~$\sE$.
\end{proof}

 The following corollary is the main result of this section.

\begin{cor} \label{anti-equivalence-corollary}
 For any lower strictly locally finite $k$\+linear category\/ $\sE$,
the vector space dualization
functor~\eqref{lf-co-contra-modules-dualization} is
anti-equivalence between the abelian category of locally finite
right\/ $\C_\sE$\+comodules and the abelian category of locally
finite left\/ $\C_\sE$\+contramodules,
$$
 (\Comodrlf\C_\sE)^\sop \xrightarrow[\,\,\N\mapsto\N^*\,]{\simeq}
 \C_\sE\Contra_\lf.
$$
\end{cor}

\begin{proof}
 Compare the descriptions of locally finite right $\C_\sE$\+comodules
and locally finite left $\C_\sE$\+contramodules provided by
Proposition~\ref{locally-finite-right-comodules}
and Theorem~\ref{locally-finite-contramodules-described}, and
take into account the anti-equivalence of module
categories~\eqref{lf-modules-dualization-equivalence} together with
commutativity of the dualization diagram~\eqref{dualization-square}.
\end{proof}

\begin{rem}
 The theory of finite-dimensional comodules and contramodules over
arbitrary coalgebras $\C$ over~$k$ is quite different from the theory
of locally finite comodules and contramodules over
the coalgebras~$\C_\sE$.

 Let $\C\comodl\subset\C\Comodl$, \ $\comodr\C\subset\Comodr\C$,
and $\C\contra\subset\C\Contra$ denote the full subcategories of
finite-dimensional comodules and contramodules.
 Then the vector space dualization functor $\N\longmapsto\M=\N^*$ 
provides an anti-equivalence between the categories of
finite-dimensional left and right $\C$\+comodules,
\begin{equation} \label{fin-dim-comodule-dualization}
 (\comodr\C)^\sop \xrightarrow[\,\,\N\mapsto\N^*\,]{\simeq} \C\comodl.
\end{equation}
 There is also the exact, fully faithful dualization functor
\begin{equation} \label{fin-dim-co-contra-modules-dualization}
 (\comodr\C)^\sop \xrightarrow{\,\,\N\mapsto\N^*\,}
 \C\contra.
\end{equation}
 Consequently, any finite-dimensional left $\C$\+comodule acquires
a left $\C$\+contramodule structure, so there is an exact, fully
faithful functor $\M\longmapsto\fP=\M$ defined on finite-dimensional
comodules~\cite[Section~A.1.2]{Psemi}, \cite[Section~8.5]{Pksurv},
\begin{equation} \label{fin-dim-comodules-to-contramodules}
 \C\comodl \xrightarrow{\,\,\M\mapsto\M\,} \C\contra.
\end{equation}

 The functors~\eqref{fin-dim-co-contra-modules-dualization}
and~\eqref{fin-dim-comodules-to-contramodules} are \emph{not}
essentially surjective for an infinite-dimensional coalgebra~$\C$
in general~\cite[Section~A.1.2]{Psemi}.
 Comparing the discussion of irreducible comodules and
contramodules in~\cite[Section~A.1.2]{Psemi},
\cite[Section~8.5]{Pksurv} with the discussion of the Ext spaces
for finite-dimensional comodules and contramodules
in~\cite[Section~A.1.2]{Psemi}, \cite[proof of Proposition~3.3]{Phff},
one can see that, for a conilpotent coalgebra $\C$,
the functors~(\ref{fin-dim-co-contra-modules-dualization}\+-%
\ref{fin-dim-comodules-to-contramodules}) are essentially surjective
if and only if the coalgebra $\C$ is finitely cogenerated
(cf.\ Remark~\ref{long-remark}).

 For comparison, \emph{no functors similar
to~\eqref{fin-dim-comodule-dualization}
or~\eqref{fin-dim-comodules-to-contramodules} exist} for locally
finite comodules and contramodules over the coalgebra $\C_\sE$
corresponding to a lower/upper strictly locally finite $k$\+linear
category $\sE$ in general, as one can easily observe in the case of
the category $\sE$ from
Example~\ref{noncomodule-noncontramodule-counterex}
and Remark~\ref{nonuniform-remark} below.
 On the other hand, Example~\ref{contra-ff-counterex} shows that
the lower strict local finiteness condition on the category $\sE$
\emph{cannot} be replaced by the weaker condition of local finiteness
of $\sE$ in Corollary~\ref{anti-equivalence-corollary}.
\end{rem}

\Section{Closedness Under Extensions}
\label{extension-closure-secn}

 We start with an obvious corollary of the arguments in
Section~\ref{locally-finite-contramodules-secn}, particularly
of Theorem~\ref{locally-finite-contramodules-described}
and Lemma~\ref{contrafinite-extension-closed}.

\begin{cor} \label{loc-fin-contramodules-extension-closed}
 Let\/ $\sE$ be a lower strictly locally finite $k$\+linear category.
 Then the class of all left\/ $\sE$\+modules of the form\/
$\Theta_\sE(\fP)$, where\/ $\fP$ ranges over the locally finite
left\/ $\C_\sE$\+contramodules, is closed under extensions in
the abelian category of (locally finite) left\/ $\sE$\+modules.
\end{cor}

\begin{proof}
 It is obvious that any extension of two locally finite $\sE$\+modules
is locally finite.
 Theorem~\ref{locally-finite-contramodules-described} tells that
a locally finite left $\sE$\+module $P$ has the form $\Theta_\sE(\fP)$
for some (locally finite) left $\sE$\+contramodule $\fP$ if and only if
$P$ is contrafinite in the sense of
Definition~\ref{contrafinite-module}.
 By Lemma~\ref{contrafinite-extension-closed}, any extension of two
contrafinite left $\sE$\+modules is contrafinite.
\end{proof}

 The aim of this section is to prove a comodule version of
Corollary~\ref{loc-fin-contramodules-extension-closed}, which does
not require the local finiteness assumption on the (co)modules.
 The following definition is a ``comodule'' version of
Definition~\ref{contrafinite-module}.

\begin{defn}
 Let $\sE$ be a locally finite $k$\+linear category and $N$ be
a right $\sE$\+module.
 We will say that $N$ is \emph{cofinite} if it satisfies the condition
of Proposition~\ref{locally-finite-right-comodules}.
 To repeat, this means that, for every object $y\in\sE$, there is
a finite subset of objects $A_{N,y}\subset\sE$ such that, for all
objects $x\in\sE$, \ $x\notin A_{N,y}$, the action map
$\Hom_\sE(x,y)\ot_k N(y)\rarrow N(x)$ vanishes.
\end{defn}

\begin{lem} \label{cofinite-extension-closed}
 Let\/ $\sE$ be a lower strictly locally finite $k$\+linear category.
 Then the class of all cofinite right\/ $\sE$\+modules is closed
under extensions in\/ $\Modr\sE$.
\end{lem}

\begin{proof}
 This is the dual version of Lemma~\ref{contrafinite-extension-closed},
provable by exactly the same argument with the action maps going in
the opposite direction.
 Alternatively, one can reduce the problem at hand to the assertion
of Lemma~\ref{contrafinite-extension-closed} using the vector space
dualization functor $N\longmapsto N^*$ of
Section~\ref{vector-space-dualization-secn}.
 One easily observes that a right $\sE$\+module $N$ is cofinite if and
only if the left $\sE$\+module $P=N^*$ is contrafinite.
 Given a short exact sequence $0\rarrow L\rarrow M\rarrow N\rarrow0$
of right $\sE$\+modules, consider the short exact sequence of left
$\sE$\+modules $0\rarrow N^*\rarrow M^*\rarrow L^*\rarrow0$.
 If the right $\sE$\+modules $L$ and $N$ are cofinite, then the left
$\sE$\+modules $N^*$ and $L^*$ are contrafinite, and
Lemma~\ref{contrafinite-extension-closed} tells that the left
$\sE$\+module $M^*$ is contrafinite.
 Hence the right $\sE$\+module $M$ is cofinite.
\end{proof}

\begin{cor} \label{loc-fin-comodules-extension-closed}
 Let\/ $\sE$ be a lower strictly locally finite $k$\+linear category.
 Then the class of all right\/ $\sE$\+modules of the form\/
$\Upsilon_{\sE^\sop}(\N)$, where\/ $\N$ ranges over the locally finite
right\/ $\C_\sE$\+comodules, is closed under extensions in
the abelian category of (locally finite) right\/ $\sE$\+modules.
\end{cor}

\begin{proof}
 Similar to the proof of
Corollary~\ref{loc-fin-comodules-extension-closed}, and based on
Proposition~\ref{locally-finite-right-comodules} and
Lemma~\ref{cofinite-extension-closed}.
\end{proof}

\begin{rem} \label{closure-properties-of-comodules-in-modules}
 Let $\sE$ be a locally finite $k$\+linear category and $M$ be
a left $\sE$\+module.
 Then $M$ arises from a left $\C_\sE$\+comodule via the comodule
inclusion functor $\Upsilon_\sE$ if and only if, for every object
$x\in\sE$, the action map
$$
 M(x)\lrarrow\prod\nolimits_{y\in\sE}\Hom_k(\Hom_\sE(x,y),M(y))
$$
factorizes through the vector subspace
$$
 \bigoplus\nolimits_{y\in\sE}\Hom_k(\Hom_\sE(x,y),M(y))\,\subset\,
 \prod\nolimits_{y\in\sE}\Hom_k(\Hom_\sE(x,y),M(y)),
$$
that is
$$
 M(x)\lrarrow\bigoplus\nolimits_{y\in\sE}\Hom_k(\Hom_\sE(x,y),M(y)),
$$
as in~\eqref{nu-x-formula}.
 It follows that the essential image of the functor
$\Upsilon_\sE\:\C_\sE\Comodl\rarrow\sE\Modl$ is closed under
subobjects, quotients, direct sums, and direct limits in $\sE\Modl$.
 Similarly, the essential image of the functor
$\Upsilon_{\sE^\sop}\:\Comodr\C_\sE\rarrow\Modr\sE$ is closed under
subobjects, quotients, direct sums, and direct limits in $\Modr\sE$.

 More generally, for any coalgebra $\C$ over~$k$ and any dense
subring $R\subset\C^*$, the essential image of the fully faithful
composition of inclusion/forgetful functors $\C\Comodl\rarrow\C^*\Modl
\rarrow R\Modl$ is closed under submodules, quotients, direct sums,
and direct limits in $R\Modl$.
\end{rem}

\begin{lem} \label{direct-union-of-locally-finite}
 Let\/ $\sE$ be a locally finite $k$\+linear category.
 Then every\/ $\sE$\+module is the direct limit of its locally finite\/
$\sE$\+submodules.
\end{lem}

\begin{proof}
 Given an object $y\in\sE$, the \emph{free right\/ $\sE$\+module with
one generator sitting at~$y$} is defined as the functor
$\sE^\sop\rarrow k\Vect$ represented by the object~$y$.
 In other words, the free right $\sE$\+module $F_y$ has the components
$F_y(x)=\Hom_\sE(x,y)$ and the action maps induced by
the multiplication/composition of morphisms in~$\sE$.
 A right $\sE$\+module is said to be \emph{finitely generated} if it
is a quotient of a finitely generated free right $\sE$\+module, i.~e.,
a quotient of a finite direct sum of right $\sE$\+modules of the form
$F_y$, \,$y\in\sE$.
 The point is that every $\sE$\+module is the direct limit of its
finitely generated $\sE$\+submodules.
 This holds over any $k$\+linear category~$\sE$.

 When $\sE$ is locally finite, all finitely generated $\sE$\+modules
are locally finite (because the vector spaces of morphisms
$\Hom_\sE(x,y)$ are finite-dimensional).
 Finally, the sum of any two locally finite submodules in any
$\sE$\+module is locally finite; so the poset of all locally finite
submodules is directed by inclusion.
\end{proof}

 Now we can prove the main result of this section.
 It is not very original (cf.~\cite[Lemma~1.2 and Theorem~4.8]{Iov}),
but we include it in this paper, together with a full proof, for
the sake of completeness of the exposition.

\begin{prop} \label{comodules-extension-closed}
 Let\/ $\sE$ be a lower strictly locally finite $k$\+linear category.
 Then the essential image of the comodule inclusion functor\/
$\Upsilon_{\sE^\sop}\:\Comodr\C_\sE\rarrow\Modr\sE$ is closed under
extensions in the category of right\/ $\sE$\+modules\/ $\Modr\sE$.
\end{prop}

\begin{proof}
 Let $0\rarrow L\rarrow M\rarrow N\rarrow0$ be a short exact sequence
of right $\sE$\+modules.
 Assume that the $\sE$\+modules $L$ and $N$ arise from right
$\C_\sE$\+comodules, that is $L=\Upsilon_{\sE^\sop}(\LL)$ and
$N=\Upsilon_{\sE^\sop}(\N)$.

 Let $M'\subset M$ be a locally finite submodule in~$M$.
 Denoting by $L'\subset L$ the full preimage of $M'$ under
the injective morphism $L\rarrow M$, and by $N'\subset N$ the image
of $M'$ under the surjective morphism $M\rarrow N$, we have a short
exact sequence of right $\sE$\+modules $0\rarrow L'\rarrow M'
\rarrow N'\rarrow0$.
 By Remark~\ref{closure-properties-of-comodules-in-modules}, both
the $\sE$\+modules $L'$ and $N'$ arise from $\C_\sE$\+comodules,
i.~e., $L'=\Upsilon_{\sE^\sop}(\LL')$ and
$N'=\Upsilon_{\sE^\sop}(\N')$ (since $L'$ is a submodule in $L$ and
$N'$ is a submodule in~$N$).

 Clearly, both the $\sE$\+modules $L'$ and $N'$ are locally finite
(since the $\sE$\+module $M'$~is).
 By Corollary~\ref{loc-fin-comodules-extension-closed}, we can conclude
that the $\sE$\+module $M'$ arises from a $\C_\sE$\+comodule, too,
that is $M'=\Upsilon_{\sE^\sop}(\M')$.
 Finally, by Lemma~\ref{direct-union-of-locally-finite},
the $\sE$\+module $M$ is the direct limit of its locally finite
submodules~$M'$.
 So Remark~\ref{closure-properties-of-comodules-in-modules} tells that
$M$ belongs to the essential image of~$\Upsilon_{\sE^\sop}$.
 Specifically, one has $M=\Upsilon_{\sE^\sop}(\M)$, where
$\M=\varinjlim_{M'\subset M}\M'$.
\end{proof}

\begin{rem}
 The exposition in this paper leaves a strong impression that
contramodules are harder to work with than comodules.
 The proof of Theorem~\ref{main-full-and-faithfulness-theorem} is
much more complicated than that of
Proposition~\ref{comodule-fully-faithful}, and requires stronger
assumptions.
 The proof of Theorem~\ref{locally-finite-contramodules-described}
is also much more complicated than that of
Theorem~\ref{locally-finite-comodules-described}, and requires
stronger assumptions.
 Finally, in this section, the assumptions of
Corollary~\ref{loc-fin-contramodules-extension-closed} are the same
as in Proposition~\ref{comodules-extension-closed}, but
the conclusion of the proposition is stronger, in that it applies
to all comodules, while the conclusion of the corollary is only
applicable to locally finite contramodules.

 Much of this feeling of easy comodules and difficult contramodules
disappears, however, when the results are put into proper context.
 It was discovered and explained in the paper~\cite{Phff} that
difficulties arising the comodule theory largely correspond to those
of the contramodule theory \emph{with the comodule phenomena
occuring in the cohomological degree\/~$1$ or\/~$2$ higher than
the contramodule ones}.
 One cohomological level up from the question of full-and-faithfulness
of an exact functor between abelian categories (that is,
full-and-faithfulness on Hom) is the question of extension closedness
of the essential image (that is, full-and-faithfulness on $\Ext^1$).

 And indeed, the results of Proposition~\ref{comodule-fully-faithful}
and Theorem~\ref{locally-finite-comodules-described} did not yet need
the lower strict local finiteness assumption on~$\sE$; they were valid
over any locally finite $k$\+linear category.
 But the resulf of Proposition~\ref{comodules-extension-closed}
already depends on the lower strict local finiteness.
 At least, it is \emph{not} true for an arbitrary locally finite
$k$\+linear category~$\sE$, as one can easily see in the case of
the category $\sE$ from Example~\ref{contra-ff-counterex}.

 For the same reason, based on the analogy with the case of conilpotent
coalgebras fully treated in the paper~\cite{Phff}, we expect that
the assertion of Corollary~\ref{loc-fin-contramodules-extension-closed}
should \emph{not} be true for locally infinite contramodules over
a lower strictly locally finite $k$\+linear category $\sE$ in general.
 See~\cite[Theorems~6.2 and~7.1 for $n=2$]{Phff} for comparison;
cf.~\cite[Remark~5.7 for $n=1$]{Phff}.
\end{rem}

\Section{Upper and Lower Locally Finite Categories}
\label{upper-lower-secn}

 We start with a simple counterexample showing that the conditions
of Theorems~\ref{locally-finite-comodules-described}
and~\ref{locally-finite-contramodules-described} are nontrivial and
do not hold automatically.
 The functors $\Upsilon_\sE$ and $\Theta_\sE$ are \emph{not}
essentially surjective in general; not even for locally finite
left modules over a lower strictly locally finite
$k$\+linear category~$\sE$.

\begin{ex} \label{noncomodule-noncontramodule-counterex}
 Let $\sE$ be the $k$\+linear category whose set of objects is
the set of all integers~$\boZ$.
 For any integers $m$, $n\in\boZ$, the $k$\+vector space
$\Hom_\sE(m,n)$ is one-dimensional if $m\le n$ and zero otherwise.
 The composition maps $\Hom_\sE(m,n)\ot_k\Hom_\sE(l,m)\rarrow
\Hom_\sE(l,n)$ are isomorphisms for all $l\le m\le n$.
 One can easily see that $\sE$ is a lower (and upper) strictly
locally finite $k$\+linear category.

 Let $N$ be the following left $\sE$\+module.
 The vector space $N(n)$ is one-dimensional for every object $n\in\sE$.
 The action map $\Hom_\sE(m,n)\ot_k N(m)\rarrow N(n)$ is an isomorphism
for all $m\le n$.
 So $N$ is a locally finite left $\sE$\+module.
 It is clear that \emph{neither} the condition of
Theorem~\ref{locally-finite-comodules-described} \emph{nor} the one of
Theorem~\ref{locally-finite-contramodules-described} is satisfied
for the $\sE$\+module~$N$.
 The ``only if'' implications of the two theorems tell that
the $\sE$\+module $N$ does \emph{not} arise either from a left
$\C_\sE$\+comodule, or from a left $\C_\sE$\+contramodule.
\end{ex}

\begin{rem} \label{nonuniform-remark}
 Just as the category $\sE$ from Example~\ref{contra-ff-counterex} is
a convienient source of simple counterexamples to statements about
locally finite $k$\+linear categories that are not lower strictly
locally finite, the category $\sE$ from
Example~\ref{noncomodule-noncontramodule-counterex} is a source of
simple counterexamples to statements about lower and upper strictly
locally finite $k$\+linear categories.
 For example, consider the following family of left modules $N_l$,
\,$l\ge0$, over the $k$\+linear category $\sE$ from
Example~\ref{noncomodule-noncontramodule-counterex}.
 The vector space $N_l(n)$ is one-dimensional for $-l\le n\le l$
and zero otherwise.
 The action map $\Hom_\sE(m,n)\ot_k N(m)\rarrow N(n)$ is
an isomorphism for all $-l\le m\le n\le l$.

 Both the conditions of
Theorem~\ref{locally-finite-comodules-described} and
Theorem~\ref{locally-finite-contramodules-described} are satisfied
for each one of the $\sE$\+modules~$N_l$; so these $\sE$\+modules
belong to the essential images of both the functors $\Upsilon_\sE$
and~$\Theta_\sE$.
 But the finite sets of objects appearing in the conditions of
the two theorems \emph{cannot be chosen uniformly} for all
the modules $N_l$ as $l$~tends to infinity.
 Thus the left $\sE$\+module $\bigoplus_{l=0}^\infty N_l$, which
is \emph{not} locally finite, belongs to the essential image of
$\Upsilon_\sE$ but does \emph{not} satisfy the condition of
Theorem~\ref{locally-finite-comodules-described}.
 Similarly, the left $\sE$\+module $\prod_{l=0}^\infty N_l$, which
is \emph{not} locally finite, belongs to the essential image of
$\Theta_\sE$ but does \emph{not} satisfy the condition of
Theorem~\ref{locally-finite-contramodules-described} (in other
words, the $\sE$\+module $\prod_{l=0}^\infty N_l$ is not
contrafinite in the sense of Definition~\ref{contrafinite-module}).
\end{rem}

 The aim of this section is to formulate simple conditions on
a locally finite $k$\+linear category $\sE$ that exclude the behavior
described in Example~\ref{noncomodule-noncontramodule-counterex}.
 In other words, we provide sufficient conditions under which either
the comodule inclusion functor $\Upsilon_\sE$, or the contramodule
forgetful functor $\Theta_\sE$ is a category equivalence.

\begin{defn} \label{upper-lower-finite}
 Let $\sE$ be a locally finite $k$\+linear category.
 We will say that $\sE$ is \emph{upper finite} if, for every object
$x\in\sE$, the set of all objects $y\in\sE$ such that $x\preceq y$
is finite.
 Dually, $\sE$ is said to be \emph{lower finite} if, for every object
$y\in\sE$, the set of all objects $x\in\sE$ such that $x\preceq y$
is finite.
 (The preorder~$\preceq$ from
Definition~\ref{morphism-preorder} on the set of all objects of
$\sE$ is presumed here.)
\end{defn}

 It is clear that any lower finite $k$\+linear category is
lower strictly locally finite in the sense of
Definition~\ref{lower-strict-locally-finite}.

\begin{prop} \label{upper-finite-comodule-equivalence}
 For any upper finite $k$\+linear category\/ $\sE$, the comodule
inclusion functor\/ $\Upsilon_\sE\:\C_\sE\Comodl\rarrow\sE\Modl$
is an equivalence of categories.
\end{prop}

 Similarly, for any lower finite $k$\+linear category $\sE$,
the comodule inclusion functor $\Upsilon_{\sE^\sop}\:\Comodr
\C_\sE\rarrow\Modr\sE$ is an equivalence of categories.

\begin{proof}
 Using the notation from
Construction~\ref{from-comodule-construction} and from
the proof of Theorem~\ref{locally-finite-comodules-described},
one can say that, for any locally finite $k$\+linear category $\sE$,
the datum of a left $\C_\sE$\+comodule $\M$ is equivalent to
that of a family of $k$\+vector spaces $\M_x$, indexed by the objects
$x\in\sE$ and endowed with the maps
$$
 \nu_x\:\M_x\lrarrow\bigoplus\nolimits_{y\in\sE}^{x\preceq y}
 \C_\sE^{x,y}\ot_k\M_y
$$
as in~\eqref{nu-x-formula}.
 When the category $\sE$ is upper finite, the direct sum in
the right-hand side of this map is finite.
 For any $k$\+linear category $\sE$, and all objects $x$, $y\in\sE$,
we have the isomorphisms $\Hom_\sE(x,y)\simeq\bigl(\C_\sE^{x,y})^*$
with finite-dimensional vector spaces~$\C_\sE^{x,y}$.
 Hence the datum of the maps~$\nu_x$ is equivalent to that of
the action maps $\Hom_\sE(x,y)\ot_k M(x)\rarrow M(y)$ for all
$x$, $y\in\sE$, where $M(x)=\M_x$.
 The coassociativity and counitality equations on the maps~$\nu_x$
are equivalent to the associativity and unitality equations on
the action maps defining a left $\sE$\+module structure on~$M$.
\end{proof}

\begin{prop} \label{lower-finite-contramodule-equivalence}
 For any lower finite $k$\+linear category\/ $\sE$, the contramodule
forgetful functor\/ $\Theta_\sE\:\C_\sE\Contra\rarrow\sE\Modl$
is an equivalence of categories.
\end{prop}

 Similarly, for any upper finite $k$\+linear category $\sE$,
the contramodule forgetful functor $\Theta_{\sE^\sop}\:\Contrar
\C_\sE\rarrow\Modr\sE$ is an equivalence of categories.

\begin{proof}
 Using the notation from
Construction~\ref{from-contramodule-construction} and from
the proof of the ``if'' implication of
Theorem~\ref{locally-finite-contramodules-described},
one can say that, for any locally finite $k$\+linear category $\sE$,
the datum of a left $\C_\sE$\+contramodule $\fP$ is equivalent to
that of a family of $k$\+vector spaces $\fP^x$, indexed by the objects
$x\in\sE$ and endowed with the maps
$$
 \pi^y\:\prod\nolimits_{x\in\sE}^{x\preceq y}
 \Hom_k(\C_\sE^{x,y},\fP^x)\lrarrow\fP^y, \qquad y\in\sE
$$
as in~\eqref{pi-y-formula}.
 When the category $\sE$ is lower finite, the direct product in
the left-hand side of this map is finite.
 Hence the datum of the maps~$\pi^y$ is equivalent to that of
the action maps $\Hom_\sE(x,y)\ot_k P(x)\rarrow P(y)$ for all
$x$, $y\in\sE$, where $P(x)=\fP^x$.
 The contraassociativity and contraunitality equations on
the maps~$\pi^y$ are equivalent to the associativity and unitality
equations on the action maps defining a left $\sE$\+module
structure on~$P$.
\end{proof}

\begin{exs}
 The following examples of locally finite $k$\+linear categories,
which are perhaps more substantial or interesting than the examples
above in this paper, are discussed at length in
the paper~\cite[Sections~6.2\+-6.4]{Psa}.

 First of all, there are three main examples of ambient $k$\+linear
categories which are \emph{not} locally finite in the sense of
our Definition~\ref{locally-finite-category}.
 These are: the Brauer diagram category $\Br(\delta)$
in~\cite[Section~6.2]{Psa}; the Temperley--Lieb diagram category
$\TL(\delta)$ in~\cite[Section~6.3]{Psa}; and the $k$\+linearization
$k[\Delta S]$ of the Reedy category $\Delta S$ of simplices in
a simplicial set $S$ in~\cite[Section~6.4]{Psa}.
 Here $\delta\in k$ is a parameter.
 All the three categories have finite-dimensional Hom spaces, so they
satisfy the condition of Definition~\ref{locally-finite-category}(i);
but the interval finiteness condition~(ii) from this definition does
\emph{not} hold for them.

 The $k$\+linear categories relevant as examples for the present are
certain subcategories (on the same set of objects) in
$\Br(\delta)$, \,$\TL(\delta)$, and $k[\Delta S]$ appearing as
components of triangular decompositions of these ambient $k$\+linear
categories.
 Following the discussion in~\cite[Section~6]{Psa}, we have:
\begin{itemize}
\item the $k$\+linear categories $\Br^+(\delta)$ and $\Br^{+=}(\delta)$
are lower finite and upper strictly locally
finite~\cite[Lemmas~6.6(a) and~6.7(a)]{Psa};
\item the $k$\+linear categories $\Br^-(\delta)$ and $\Br^{=-}(\delta)$
are upper finite and lower strictly localy
finite~\cite[Lemmas~6.6(b) and~6.7(b)]{Psa};
\item the $k$\+linear category $\TL^+(\delta)$ is lower finite and
upper strictly locally finite~\cite[Lemmas~6.10(a) and~6.11(a)]{Psa};
\item the $k$\+linear category $\TL^-(\delta)$ is upper finite and
lower strictly locally finite~\cite[Lemmas~6.10(b) and~6.11(b)]{Psa};
\item the $k$\+linear category $k[\Delta^-S]$ is upper finite and
lower strictly locally finite~\cite[Lemmas~6.15(b) and~6.16(b)]{Psa};
\item the $k$\+linear category $k[\Delta^+S]$ is lower
finite~\cite[Lemma~6.15(a)]{Psa};
\item for a simplicial set $S$ whose set of $n$\+simplices $S_n$ is
finite for every $n\ge0$, the $k$\+linear category $k[\Delta^+S]$
is upper strictly locally finite~\cite[Lemma~6.16(a)]{Psa}.
\end{itemize}
\end{exs}

\end{document}